\documentclass[11pt]{amsart}
\usepackage{amsfonts}
\usepackage{amsmath}
\usepackage{amssymb}
\usepackage{mathtools}
\usepackage{graphicx, float}
\usepackage[hmargin=2.8cm,vmargin=2.3cm]{geometry}
\usepackage[utf8]{inputenc}
\usepackage{enumitem}
\usepackage{xcolor}

\graphicspath{{Figures/}}

\newtheorem{definition}{Definition}[section]
\newtheorem{remark}[definition]{Remark} 
\newtheorem{proposition}[definition]{Proposition}
\newtheorem{theorem}[definition]{Theorem}

\newtheorem{corollary}[definition]{Corollary}
\newtheorem{lemma}[definition]{Lemma}

\newtheorem{example}[definition]{Example}

\numberwithin{equation}{section}

\DeclareMathOperator{\Id}{Id}
\DeclareMathOperator{\re}{Re}

\DeclareMathOperator{\diverg}{div}
\DeclareMathOperator{\SDiff}{SDiff}
\DeclareMathOperator{\SoDiff}{S_0Diff}
\DeclareMathOperator{\SVect}{SVect}
\DeclareMathOperator{\SoVect}{S_0Vect}
\DeclareMathOperator{\ad}{ad}
\DeclareMathOperator{\Ad}{Ad}

\DeclareMathOperator{\RCT}{R}
\DeclareMathOperator{\sg}{\rm sgrad}

\DeclareMathOperator{\Ric}{Ric}

\newcommand{\p}{\partial}
\newcommand{\mb}{\mathbb}
\newcommand{\mf}{\mathfrak}
\newcommand{\norm}[1]{\left\lVert#1\right\rVert}
\newcommand{\mcE}{\mathcal{E}}

\title[Curvatures of SDiff$(\mathbb{K})$ and SDiff$(\mathbb{RP}^2)$. ]{Curvatures of  Measure-Preserving Diffeomorphism Groups of Non-orientable Surfaces}
\author{Boris Khesin \and René Langøen \and Irina Markina}

\address{B.K.: Department of Mathematics, University of Toronto, Canada}

\address{R.L.: Department of Mathematics, University of Bergen, Norway.}

\address{I.M.: Department of Mathematics, University of Bergen, Norway.}

\usepackage[pdfencoding=auto, psdextra]{hyperref}

\usepackage[ ]{amsrefs}
\begin{document}

\begin{abstract}
We study curvatures of the groups of measure-preserving diffeomorphisms of non-orientable compact surfaces.
For the cases of the Klein bottle and the real projective plane we compute curvatures, their asymptotics and the normalized Ricci 
curvatures in many directions. Extending the approach of V.~Arnold, and A.~Lukatskii we provide 
estimates of weather unpredictability for natural models of trade wind currents on the Klein bottle and the projective plane.
\end{abstract}

\maketitle

\tableofcontents

\section{Introduction} 
In the 1960's, V. Arnold pioneered a geometric approach for studying hydrodynamics~\cite{arnold_sur_1965, arnold_sur_1966}. Recall that the motion of an inviscid incompressible  fluid filling an $n$-dimensional Riemannian manifold  
$M$  is governed by the hydrodynamic Euler equation
$$
\partial_t v+\nabla_v v=-\nabla p
$$
on  the divergence-free velocity field $v$ of a fluid flow in $M$. 
Here $\nabla_v v$ stands for  the Riemannian 
covariant derivative of the field $v$ along itself, while the function $p$  is determined by the divergence-free 
condition ${\rm div}\,v=0$ up to an additive constant. 
Arnold proved that solutions to the Euler equation describe geodesic curves in the infinite-dimensional group $\SDiff(M)$
of volume-preserving diffeomorphisms of the  manifold $M$ with respect to the right-invariant $L^2$-metric given by the fluid's kinetic energy.
In particular, geometric properties of the corresponding infinite-dimensional Riemannian manifold $\SDiff(M)$, such as its curvatures, conjugate points, etc., affect the dynamical behaviour of geodesics on it, and hence the corresponding Euler solutions.

The corresponding groups of measure-preserving diffeomorphisms can be considered for non-orientable Riemannian 
manifolds $M$ as well.  It turns out that the corresponding Euler equations can be defined in the non-orientable context with minimal adjustments, see \cite{balabanova2022hamiltonian1, vanneste2021vortex, Izosimov2023}. 
More precisely, while a non-orientable manifold does not admit a non-vanishing volume form, one can still define a measure on it as a density (also called volume pseudo-form), see e.g.~\cite{lee_introduction_2012}.
We denote by  $\SDiff(M)$ the infinite-dimensional group of measure-preserving diffeomorphisms of $M$, isotopic to the identity. 
Many questions about those groups for non-orientable manifolds $M$ can be addressed by considering them for the corresponding orientable cover $\tilde M$.

In this paper we are focusing on non-orientable surfaces, and in particular on the Klein bottle and projective plane. Their double orientation covers are the torus and the sphere, respectively. For the torus, Arnold computed sectional curvatures in~\cite{arnold_sur_1966}. Those results were generalized to the $n$-torus in~\cite{lukatskii_curvature_1984}, where also an infinite-dimensional version of the normalized Ricci curvature was introduced and computed. The sectional curvatures of the diffeomorphism group of the 2D sphere in some selected directions were computed in~\cite{lukatskii_curvature_1979}. 

\smallskip

Observe that the subgroup of diffeomorphisms that commute with any isometry of a Riemannian manifold $\tilde M$
is  a totally geodesic subgroup of $\SDiff({\tilde M})$ for the right-invariant $L^2$ metric (Proposition \ref{prop:isometry}). This allows one to study curvatures of diffeomorphism groups of non-orientable manifolds via their
lifts to the orientable covers. The corresponding Lie algebras of divergence-free  vector fields on non-orientable $M$ 
are $I$-invariant subalgebras in the algebras of vector fields in $\tilde M$, where $I$ is the orientation-reversing isometry of $\tilde M$ without fixed points.

Our main objects of study are diffeomorphism groups of non-orientable surfaces. While exact
divergence-free  vector fields on surfaces are described by their stream (or Hamiltonian) functions, for non-orientable surfaces 
one has to confine to $I$-anti-invariant stream functions on their covers. 
We start with the case of the Klein bottle $\mb K$, whose orientation cover is the flat torus $\mb T$.

In Section \ref{sec:klein.bottle.case} we first describe a special basis of vector fields $\xi_k$, $k=(k_1,k_2)\in \mb N^2$ 
spanning the corresponding Lie algebra $\SVect(\mb K)$ of divergence-free fields on $\mb K$, see Theorem~\ref{th:basis.g-1}.
Then, similarly to ~\cite{arnold_sur_1966}, we focus on sectional curvatures in two-dimensional planes containing a specific 
vector field, mimicking a trade wind. The corresponding curvatures are described in Theorem \ref{th:sec.curv.klein}.
One of interesting conclusions of those explicit computations is the following
\smallskip

{\bf Corollary \ref{cor:neg.C.basis}}. {\it If two stream functions $\xi_k, \, \eta_l$  on the Klein bottle $\mb K$ are sufficiently different, in the sense that  $k_1\neq l_1$, $k_2\neq l_2$, then the sectional curvature $C(\xi_k, \eta_l)$  in the plane through them is strictly negative.}

On the other hand, there always exists a vector field $\eta\in \SVect(\mb K)$ such that the corresponding sectional curvature is positive, $C(\xi_k, \eta) >0$, see 
Theorem \ref{th:pos.sequence}.

Another result worth mentioning is an explicit value of the normalized Ricci curvature in the direction
of an arbitrary basis vector:
\smallskip

{\bf Theorem \ref{th:ric.klein}}. {\it For any basis vector $\xi_k\in \SVect(\mb K)$  
the normalized Ricci curvature of the group $\SDiff(\mb K)$ in the direction $\xi_k$ is negative and, explicitly,
$\Ric(\xi_k) =  -3(k_1^2 + k_2^2)/({16S_{\mb K}} ), $
where $S_{\mb K}$ is the area of the Klein bottle.} 

We also present results on asymptotics of sectional curvatures $C(\xi_k, \eta_m)$ as $m\to \infty$ and show that they are mostly negative, see Theorem \ref{th:sec.curv.large.l}. 
\medskip

In Section \ref{sec:real.projective.plane} we study curvatures of the group of measure-preserving 
diffeomorphisms of $\mb{RP}^2$. The basis of the corresponding stream functions is formed by  a ``half" of the spherical 
harmonics. We are particularly interested in the vector field $e^0_3 = {\rm const}\cdot (5\cos^2(\theta) -1)\partial_\phi $
in spherical coordinates, which is a natural candidate for a trade wind  current. 
Theorem~\ref{th:sec.curv.cpe} gives a detailed description of the corresponding sectional curvatures $C(e^0_3, e^m_l)$ in two-dimensional planes containing $e^0_3$, while the corresponding asymptotics as $l\to \infty$ for a fixed limit $m/l\to q$ 
are described in Theorem \ref{th:asymp.sec.curv.cpe}.

This allows one to describe the corresponding Ricci curvature in this direction:
\smallskip

{\bf Corollary \ref{cor:ricci.rp2}}. {\it The normalized Ricci curvature of  $\SDiff(\mb{RP}^2)$ in the direction $e^0_3$ is $\Ric(e^0_3) = -15/(8\pi) = -15/(4 S_{\mb{RP}^2 }).$}

We also present a comparison with the previously known results for the vector field $e^2_0 ={\rm const}\cdot  \cos\theta \,\p_\phi$, cf.~\cite{lukatskii_curvature_1979}.
\medskip

Finally, in Section \ref{sect:weather} we give estimates for weather unpredictability based on the curvature computations
if the earth were of the shape of the Klein bottle or the projective plane, and compare them with those for the earth shaped as a sphere or a torus, cf.~\cite{arnold_mathematical_1989}. Here it is  important to agree on which vector field could be a most natural analogue of the trade wind, a strong west-east current in the earth atmosphere, responsible, in particular, for the shorter time of flying in the east direction vs. the west one in either hemisphere. The other important point is how to estimate the corresponding ``average sectional curvature'', and we assume it to be the normalized Ricci curvature  in that direction. 

Here we encounter a new phenomenon for non-orientable manifolds: a typical error increase may depend on whether we
consider the wind's fastest particles along closed orbits where the manifold's orientation changes or does not, and on how we rescale the manifold. 
Under certain assumptions we show that to predict the weather on the Klein bottle (or on the torus) for 2 months for a natural trade wind, one needs to know it today with $4$ more digits of accuracy. On the other hand, on the projective plane for the same period  of 2 months one needs to know it today with $8$ more digits of accuracy! Computations in~\cite{arnold_mathematical_1989} and~\cite{lukatskii_curvature_1979} based on different trade wind candidates and ``average sectional curvature'' estimates would give, respectively $5$ more digits of accuracy for the torus and $10$ more digits for the sphere.
It seems that the weather is far more reliable on the Klein bottle!

\medskip

The study of non-orientable surfaces has applications in theoretical physics for toy-models of space-time universe with a non-orientable space surface. For instance in~\cite{witten_parity_2016}  Witten discusses a ``parity'' anomaly appearing under reflection symmetry in (2+1) dimensional topological gravity. Another paper~\cite{chen_quantizing_2014} quantizes BF-theory on non-orientable surfaces and, in particular, considers the Klein bottle in greater detail. In general one can speculate whether the space manifold of the space-time universe is orientable or not. Orientation is a global topological property of a manifold, and travelling around the universe along an orientation reversing path is an unfeasible experiment. Instead one might explore the implications of a non-orientable space manifold to more local structures that can be verified experimentally. 
Note that fluid dynamics might be a model for such experiments on non-orientable surfaces, as for instance soap films can have a shape of a M\"obius band. 

\medskip

{\bf Acknowledgements:}  B.K. was partially supported by an NSERC research grant. I.M and R.L are grateful for the support granted by the project Pure Mathematics in Norway, funded by the Trond Mohn Foundation and Tromsø Research Foundation. All authors are also grateful to the Mittag-Leffler Institute in Stockholm, where this work was conceived.



\section{Measure-preserving diffeomorphism groups}\label{sec:measure.preserving}
\subsection{Definitions and preliminaries}

Let $M$ be a compact, non-orientable manifold without boundary and $\tilde{M}$ an orientation double cover of $M$. Then  there exists a fixed-point-free, orientation-reversing involution $I:\tilde{M} \to \tilde{M}$ such that $ M \simeq \tilde{M}/ I$. The quotient induces a covering map $\mf K:\tilde{M}\to M$. For such a double cover the deck transformation group consists of the two elements $\{ \Id_{\tilde{M}}, I \}$.

\begin{definition}\label{def:density}
{\rm
A {\it density} $\rho$ on an $n$-dimensional manifold $M$ is a section of the bundle $\bigwedge^n T^*M \otimes o(M)$, where $o(M)$ is the orientation bundle over $M$. 
}\end{definition}
A density is also referred to as a non-vanishing volume pseudo-form, or a twisted volume-form, on $M$.
Under a change of coordinates $(x_1, \dots, x_n) \mapsto  (y_1, \dots, y_n)$ with Jacobi matrix
$J = \Big(\frac{\p x_i}{\p y_j}\Big)$ one obtains,
\[\rho \left(\frac{\p}{\p{x_1},} \dots, \frac{\p}{\p{x_n}} \right) =  |\det(J)|\rho \left(\frac{\p}{\p{y_1},} \dots, \frac{\p}{\p{y_n}} \right).  \]
Given a density $\rho $ on $M$ one can define a measure $\mu$ by
\[\mu(A) = \int_{A} \rho\quad \text{for Borel subsets}\quad A\subset M. \]

Let $\tilde{M}$ be an orientation double cover of $M$ with covering map $\mf K$ and
$$\rho:M\to {\bigwedge}^2 T^*M \otimes o(M),$$
a density on $M$. The density $\rho$ can be pulled back to $\tilde{\rho} = \mf K^*\rho  $ on $\tilde{M}$, 
while $\tilde{\rho}$ can be identified with a volume form on $\tilde{M}$.

\begin{example}\label{def:klein.bottle}
{\rm
Consider $\mathbb{R}^2$ with the lattice $\Gamma$ spanned by $( 2\pi,0)$ and $ (0, 2\pi)$
\[\Gamma = \{ (2\pi a, 2\pi b)  \in \mathbb{R}^2 \mid a,b \in \mathbb{Z} \} .\]
The torus $\mb T$ is defined by the quotient $\mathbb{R}^2/\Gamma$, see Figure~\ref{fig:lattice}.
Define the fixed-point-free, orientation-reversing involution by $I:\mathbb T \to \mathbb T$ by
\[I(x_1,x_2) = (2\pi-x_1, \pi + x_2)\] and the equivalence relation $(x_1,x_2) \sim_\mathbb{K} I(x_1, x_2)$ on $\mathbb T$. 
The Klein bottle $\mathbb{K}$ is the quotient $\mathfrak{K}:\mathbb T\to \mathbb K=\mathbb{T}/\sim_\mathbb{K}$ endowed with the smooth structure, see Figure~\ref{fig:lattice}. 
On the torus $\mb T$ we define the area form $\tilde{\rho} = dx_1 \wedge dx_2$ by restricting the standard area form on $\mb R^2$. The density $\rho$ on $\mb K$ is such that its pull-back to $\mb T$ by $\mf K$ coincides with $\tilde{\rho}$.  The covering map $\mf K$ is also required to be a Riemannian cover, where on $\mb T$ we have the standard Euclidean metric $\tilde{g} = dx_1^2 + dx_2^2$.

\begin{figure}
\centering
\includegraphics[width=0.4\textwidth]{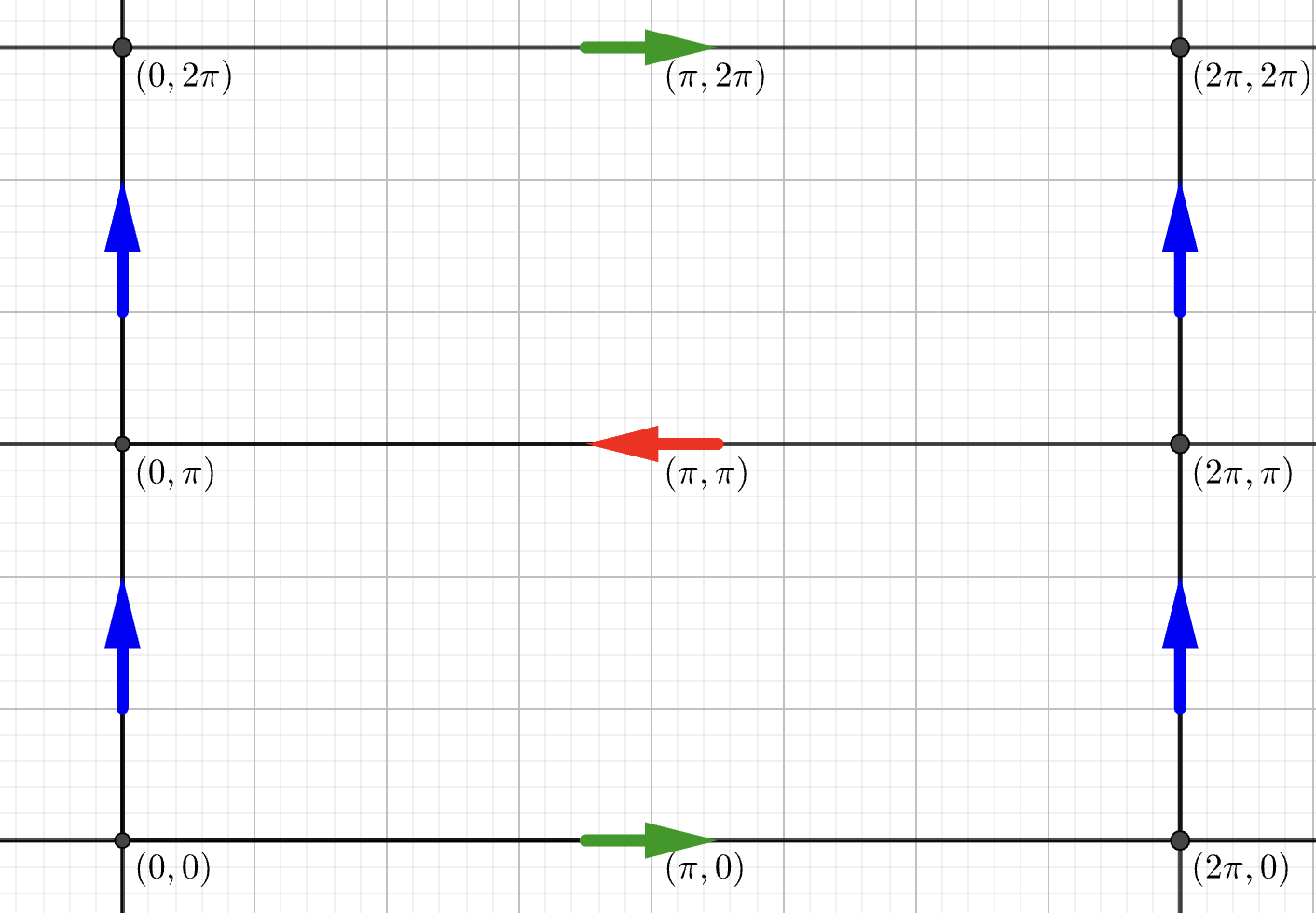}
\caption{ \tiny Lattice $\Gamma$ in $\mathbb{R}^2$. The green and blue arrows are the identifications defining the torus $\mb T$. Identifying the red arrow with the green arrows gives the Klein bottle $\mb K$, whose double cover is the torus $\mb T$.}
\label{fig:lattice}
\end{figure}
}
\end{example}
\begin{example}\label{ex:Usphere}
{\rm
The unit sphere $\mb S^2=\{(x,y,z)\in\mathbb R^3\vert\ x^2+y^2+z^2=1\}$ with the spherical coordinates \begin{equation}\label{eq:sphere}
\begin{array}{llll}
x=\sin\theta\cos\phi,
\\
y=\sin\theta\sin\phi,
\\
z=\cos\theta,
\end{array}
\qquad\theta\in[0,\pi],\ \phi\in[0,2\pi), 
\end{equation}
is another key example.
The antipodal map $I: \mb S^2 \to\mb  S^2$ given by
\begin{equation}\label{eq:InvSphere} 
I(x,y,z) = (-x, -y, -z)\quad \text{or} \quad
I(\theta, \phi) = ( \pi-\theta, \pi+\phi  )
\end{equation}
is a fixed-point-free, orientation-reversing involution. Define the 
equivalence relation $(x,y,z) \sim I(x,y,z)$ on $\mb S^2$.
The real projective plane $\mathbb{RP}^2$ is the quotient $\mathfrak{K}: \mb S^2 \to \mathbb{RP}^2 = \mb S^2/\sim$. 
Points on $\mb{RP}^2$ are given spherical coordinates $(\theta, \phi)$, through the quotient map $\mf K$. 
The area form $\tilde{\mu} = \frac{1}{|z|} dx\wedge dy $  on the sphere $\mb S^2$ (or in spherical coordinates $\tilde{\mu} = \sin(\theta) d\theta \wedge d\phi$) induces a density $\mu$ on $\mb{RP}^2$ by the covering map $\mf K$.  The covering map $\mf K$ is also required to be a Riemannian cover, where on $\mb S^2$ we have the metric $\tilde{g} = d\theta^2 + \sin^2\theta d\phi^2$ induced from the ambient Euclidean metric in $\mb R^3$. 
}
\end{example}

\subsection{Measure-preserving groups and their Lie algebras}

\begin{definition}\label{def:diff.group.non.orientable}
{\rm
Let $N$ be a compact (oriented or not) manifold without boundary and $\rho$ a density on $N$. The {\it group of measure-preserving diffeomorphisms of} $(N, \rho)$ is the Lie-Fr\'echet  group
\[ \SDiff(N) = \left\{ \Phi\colon N\to N \mid \Phi \text{ is a diffeomorphisms and } \Phi^*\rho = \rho \right\} . \]
}
\end{definition}

\begin{lemma}\label{lem:lie.group.iso}
Let $\SDiff_I(\tilde{M})$ denote the subset of $\SDiff(\tilde{M})$ consisting of diffeomorphisms that commute with the involution map $I$.
Then $\SDiff_I(\tilde{M})$ is a subgroup of $\SDiff(\tilde{M})$ and there exists an isomorphism  
\[\begin{matrix}
F: & \SDiff_I(\tilde{M}) & \to & \SDiff(M),
\end{matrix}\]
that identifies measure-preserving diffeomorphisms of a non-orientable manifold $M$ with the measure-preserving and involution-commuting diffeomorphisms of its orientation double cover~$\tilde{M}$.
\end{lemma}
\begin{proof}
The proof is immediate by the lifting properties through covering maps. See~\cite{omori_infinite_1974} for the topological properties of subgroups. 
\end{proof}

Consider the Lie-Fr\'echet group $\SDiff(\tilde{M})$ of measure-preserving diffeomorphism of $\tilde{M}$ and its Lie-Fr\'echet algebra $\SVect(\tilde{M})$ consisting of vector fields on $\tilde M$ that are divergence-free with respect to the volume form $\tilde{\rho}$. Thus
\[\SVect(\tilde{M}) = \left\{ v:\tilde{M} \to T\tilde{M} \mid L_v\tilde{\rho} = 0  \iff \diverg(v) = 0 \right\}.\]
Lemma~\ref{lem:lie.group.iso} shows that for $\tilde{\Phi}\in \SDiff(\tilde{M})$ to induce a measure-preserving diffeomorphism of $M$, it needs to satisfy
\begin{equation}\label{eq:AdI} 
\tilde{\Phi} = I \circ \tilde{\Phi} \circ I^{-1}. 
\end{equation}
For $v\in \SVect(\tilde{M})$ the condition~\eqref{eq:AdI} translates to
$ \Ad_I(v) = v$,  
which reads as $I_*v = v\circ I$.

The Lie-Fr\'echet algebra $\SVect_I(\tilde{M})$ of $\SDiff_I(\tilde{M})$ is the subalgebra of $\SVect (\tilde{M})$ defined by
\[\SVect_I(\tilde{M}) = \{v\in \SVect(\tilde{M}) \mid L_v \tilde{\rho} = 0,\quad I_*v = v\circ I\}. \]

\begin{corollary}\label{cor:iso_lie_alg}
There exists an isomorphism of Lie-Fr\'echet algebras
\[F_*: \SVect_I(\tilde{M}) \to \SVect(M),\]
induced by the isomorphism $F$ in Lemma~\ref{lem:lie.group.iso}.
\end{corollary}


\section{Curvatures of SDiff$(M)$, isometries, and stream functions}\label{sec:riemannian.structure}
\subsection{Riemannian structure and curvatures}\label{sec:curvature.def}
Let $M$ be a compact finite-dimensional (oriented or not) manifold with a density $\rho$ and a Riemannian metric $g$. 
A right-invariant Riemannian metric on $\SDiff(M)$ is defined at $\Phi \in \SDiff(M)$ by
\[
\langle v, w \rangle _\Phi =  R^*_{\Phi^{-1}} \langle v, w\rangle_{\Id}=\int_{M} g(v\circ \Phi^{-1},w\circ \Phi^{-1})\, \,{\rho}, \quad  v,w \in T_\Phi\SDiff(M).  \]
where $R^*_{\Phi^{-1}}$ is the pull-back by the right translation $R_{\Phi^{-1}}$.

For the right-invariant Riemannian metric $\langle.\,,.\rangle $ on $\SDiff(M)$ 
the Levi-Cevita connection $\nabla$ on $\SDiff(M)$ satisfies the Koszul formula
\[ 2\langle \nabla_u v ,  w\rangle  = u \langle v, w\rangle + v \langle w , u\rangle  - w\langle u, v\rangle  +\langle[u,v] ,w \rangle - \langle [v, w], u \rangle    + \langle [w, u] , v\rangle,  \]
where $[.\,,.]$ is the corresponding Lie bracket. Since  $\langle.\,,. \rangle$ induces an isomorphism of $\SVect(M)$ and its smooth dual $\SVect(M)^*$, the Koszul formula can be simplified to
\begin{equation}\label{eq:Koszul.lie}
\nabla_u v = \frac{1}{2} \left( [u,v] - B(u,v) - B(v,u) \right), 
\end{equation}
where $u$ and $v$ are right invariant vector fields on $\SDiff(M)$ and $B$ is the bilinear operator implicitly defined by
\[ \langle B(u,v) , w \rangle = \langle u  , [v,w] \rangle  = \langle \ad_v^T(u),  w\rangle ,\]
see~\cite{arnold_mathematical_1989}(p. 333).

In the cases when $B$ is well defined, there exists a Levi-Cevita connection on $\SDiff(M)$, i.e the unique connection satisfying
$$
\text{i)}\quad u\langle v ,w \rangle  = \langle \nabla_u v , w\rangle + \langle v , \nabla_u w\rangle \quad
\text{ii)} \quad \nabla_u v - \nabla_v u - [u,v] = 0.
$$
The \textit{Riemannian curvature endomorphism} $R$ and the \textit{sectional curvature} $C$ of $\SDiff(M)$ is defined similarly to a finite dimensional Riemannian manifold:
\begin{equation}\label{eq:Rie.curv.end.def}
R(u,v)w = -\nabla_u\nabla_v w + \nabla_v \nabla_u w + \nabla_{[u,v]} w \quad \text{for } u,v,w \in T_\Phi \SDiff(M)\,,
\end{equation}
\begin{equation}\label{eq:sec.curv.def}
C(u,v) = \frac{\langle  R(u,v)u, v \rangle }{\langle u,u \rangle \langle v, v\rangle - \langle u, v \rangle^2   } \quad \text{for } u,v \in T_\Phi \SDiff(M).
\end{equation}

Another formula for $C(u,v)$ in the direction determined by an orthonormal pair of vectors $u,v \in \SVect(M)$ is provided by Arnold in~\cite{arnold_mathematical_1989}: 
\begin{equation}\label{eq:C.short.formula}
C(u,v)=\langle \delta(u,v),\delta(u,v)\rangle+2\langle\alpha(u,v),\beta(u,v) \rangle-3\langle\alpha(u,v),\alpha(u,v) \rangle,
\end{equation}
where
\begin{equation}\label{eq:ABD}
\begin{array}{lll}
\delta(u,v) & = & \frac{1}{2}\Big(B(u,v)+B(v,u)\Big), \\
\beta(u,v)  & = & \frac{1}{2}\Big(B(u,v)-B(v,u)\Big), \\
\alpha(u,v) & = & \frac{1}{2}[u,v].
\end{array}
\end{equation}

\subsection{Isometries and the totally geodesic property}\label{sec:totally}
Geodesics on $ \SDiff(M)$ with respect to the right-invariant $L^2$-metric are of particular interest, as they correspond to solutions of the Euler equation on $M$, as was shown by Arnold~\cite{arnold_sur_1966}.

\begin{definition}
{\rm A submanifold $N$ of a Riemannian manifold $(M, g)$ is called {\it totally geodesic} if any geodesic on the submanifold $N$ with its induced Riemannian metric $g$ is also a geodesic on the Riemannian manifold $(M,g)$.

Sectional curvatures of totally geodesic submanifolds $N$ coincide with the curvatures in $M$ in the corresponding two-dimensional directions tangent to $N$. 
}
\end{definition}

Now consider an isometry $\tau$ of a Riemannian manifold $(M,g)$. The following general proposition 
will be particularly useful in the non-orientable context later.

\begin{proposition}\label{prop:isometry}
Let $\SDiff_\tau(M)$ denote the subset of $\SDiff({M})$ consisting of diffeomorphisms that commute with the isometry $\tau$.
Consider the right-invariant $L^2$ metric on $\SDiff({M})$.
Then $\SDiff_\tau({M})$ is a totally geodesic subgroup of $\SDiff({M})$. 
\end{proposition}

\begin{proof}
Indeed, the compositions and inverses of diffeomorphisms  commuting with $\tau$ also commute with it, so 
$\SDiff_\tau({M})$ is a subgroup.  The Koszul formula \eqref{eq:Koszul.lie} for the covariant derivative uses 
only the metric and Lie commutator, both of which are $\tau$-invariant, so the parallel transport is $\tau$-invariant as well. 
The latter means that all geodesics started on the submanifold $\SDiff_\tau({M})\subset \SDiff({M})$ remain on this submanifold, which is the  totally geodesic property.
\end{proof}

A similar statement for axisymmetric flows was proved in~\cite{lichtenfelz_axisymmetric_2022}.

\begin{corollary} For a non-orientable manifold $M$ the group of measure-preserving diffeomorphisms $\SDiff(M)$, understood as a subgroup of $I$-invariant diffeomorphisms in its orientation double cover $\SDiff_I(\tilde{M}) \subset \SDiff(\tilde{M}) $,
is a totally geodesic submanifold in $\SDiff(\tilde{M}) $.
\end{corollary}

Indeed, the involution $I$ is an isometry of $M$.

This allows one to employ the curvature formulas for the diffeomorphism groups of the corresponding 
orientation double covers, provided that the basis vector fields are chosen in invariant or anti-invariant way.
Furthermore, below we assume  that for a non-orientable $M$ and its orientation cover $\tilde M$  the Lie group isomorphism $F:\SDiff_I( \tilde{M}) \to \SDiff(M)$ (from Lemma \ref{lem:lie.group.iso}) is also {\it an isometry} of the right-invariant Riemannian metrics $\langle \cdot, \cdot \rangle_{\tilde{M}} $ and $\langle \cdot, \cdot \rangle_{M}$ on $\SDiff_I( \tilde{M}) $ and $ \SDiff(M)$, respectively. In other words, we set the norm of a vector field $v \in \SVect(M)$ to be that of its lift 
$\tilde{v}\in \SVect_I(\tilde{M})$: 
$$
\norm{v}^2_{{M}} :=\int_{\tilde{M}} {g}(\tilde{v}, \tilde{v}) \, \tilde{\rho}  = 2\int_{{M}} g({v},  {v}) \, \rho = \norm{\tilde{v}}^2_{\tilde{M}}\,. 
$$
In this isometric normalization sectional curvatures of the group $\SDiff(M)$ and the subgroup $\SDiff_I(\tilde{M}) \subset \SDiff(\tilde{M})$ coincide:  $C_{M}(u,v)=C_{\tilde{M} } (\tilde{u}, \tilde{v})$.

\begin{remark}\label{rem:norm}
{\rm
Note that in the normalization $\norm{v}^2_{{M}} := \int_{{M}} g({v},  {v}) \, \rho = \frac{1}{2} \norm{\tilde{v}}^2_{\tilde{M}} ,$
the sectional curvatures of the group $\SDiff(M)$  and the subgroup $\SDiff_I(\tilde{M}) $ are related by a factor of 2, namely 
$C_{M}(u,v)=2C_{\tilde{M} } (\tilde{u}, \tilde{v})$, as the formula \eqref{eq:sec.curv.def} implies.
}
\end{remark}
\subsection{Stream functions}\label{sec:stream.func}

Let $\tilde{M}$ be a compact oriented surface and $\omega$ an area form on $\tilde{M}$. For a vector field $v\in \SVect(\tilde{M})$ consider the  1-form $\alpha = \omega(v, \cdot )$. 
The 1-form $\alpha$ is closed since $v$ is divergence-free. 
Consider a field $v$ for which it is an exact 1-form, i.e., the differential of a Hamiltonian function,  $df_v\coloneq\omega(v,\cdot )$. This Hamiltonian function $f_v$ (defined up to an additive constant) on $M$ is called the stream function of $v$. 
We restrict the Lie algebra $\SVect(\tilde{M})$ to the subalgebra $\SoVect(\tilde{M})$ of vector fields $v$ admitting a single-valued, zero-mean stream function.
The correspondence $v \leftrightarrow f_v$ is a Lie algebra isomorphism between $\SoVect(\tilde{M})$ and the Poisson algebra \[ C_0^{\infty}(\tilde{M})  = \left\{ f : \tilde{M}\to \mb R \mid f \text{ is smooth}, \, \int_{\tilde{M}} f\omega=0  \right\} , \quad \left\{ f,g \right\} = dg(v_f) ,\]
where $v_f=\sg f$ is the skew-gradient of $f$, such that $df(\cdot) =\omega(v_f, \cdot )$.


\begin{example}
{\rm Let $\tilde{M} = \mb T$ be the torus and $\tilde{\rho} = dx_1 \wedge dx_2 $ the area form on $\mb T$. The area form is 
a symplectic form on $\mb T$. A vector field $v = v_1\p_{x_1} + v_2 \p_{x_2} \in \SVect(\mb T)$ can be identified with a closed 
1-form $ \tilde{\rho}(v, \cdot ) = -v_2 dx_1 + v_1dx_2. $
Let $\SoVect(\mb T)$ be the space of exact divergence-free vector fields on the torus. They are associated to the space $\mf g^{\mb T}$ of zero-mean stream functions:
\begin{equation}\label{eq:stream.func}
\begin{gathered}
\SoVect (\mb T) =  \left\{ v_f = \frac{\p f}{\p x_2}\p_{x_1}  - \frac{\p f}{\p x_1}\p_{x_2}  \, \Big| \,  L_{v_f}\tilde{\rho}= \text{div}(v_f) \tilde{\rho} = 0  \right\}    \\
\longleftrightarrow \quad \mf g^{\mb T} \coloneq \left\{ f:\mathbb{T} \to \mathbb{R}\mid df =  \tilde{\rho}(v_f, \cdot  ) =- v_2dx_1 + v_1dx_2 , \quad \int_\mathbb{T}f \, \tilde{\rho} = 0 \right\}.
\end{gathered}
\end{equation}
This is a Lie algebra isomorphism, where the Lie bracket of vector fields corresponds to the canonical Poisson bracket of stream functions:
\[  \left[v_f, w_g\right] = -(v_f w_g - w_g v_f )  \quad \longleftrightarrow \quad  \left\{f, g\right\} = \frac{\p f}{\p x_2}\frac{\p g}{\p x_1} - \frac{\p f}{\p x_1}\frac{\p g}{\p x_2}  .\]
}
\end{example}

\begin{example}
{\rm 
Let $\tilde{M} = \mb S^2$ and consider the area form $\mu = \sin \theta \, d\theta \wedge d\phi$, which is also a symplectic form on $\mb S^2$. A vector field $v(\theta, \phi) = v_\theta \p_\theta + v_\phi \p_\phi\in \SVect(\mb S^2)$ can be identified with a closed 1-form $\mu(v, \cdot ) = \sin \theta(- v_\phi d\theta +  v_\theta d\phi ).$ On  $\mb S^2$ all closed 1-forms are exact and
a field $v$ corresponds to a zero-mean stream function $f_v:\mb S^2 \to \mb R$. 	The corresponding Lie algebra isomorphism of vector fields and their stream functions
\begin{equation}\label{eq:stream.func.sphere}
\begin{gathered}
\SVect (\mb S^2) =  \left\{ v_f = \frac{1}{\sin \theta}\frac{\p f}{\p \phi} \p_\theta  -\frac{1}{\sin \theta} \frac{\p f}{\p \theta}\p_\phi \, \Big| \,  L_{v_f}\mu= \text{div}(v_f) \mu = 0  \right\}    \\
\longleftrightarrow \quad \mf g^{\mb S^2} \coloneq \left\{ f: \mb S^2 \to \mathbb{R}\mid df = \mu(v_f, \cdot  ) = \sin\theta (-v_\phi d\theta + v_\theta d\phi) , \int_{\mb S^2} f \mu = 0 \right\}
\end{gathered}
\end{equation}
is as follows:
\begin{equation}\label{eq:poisson.sphere}
\left[v_f, w_g\right] = -(v_f w_g - w_g v_f )  \quad \longleftrightarrow \quad  \left\{f, g\right\} = \frac{1}{\sin\theta} \left(\frac{\p f}{\p \phi}\frac{\p g}{\p \theta} - \frac{\p f}{\p \theta}\frac{\p g}{\p \phi}\right).
\end{equation}
}
\end{example}

\begin{remark}\label{rem:stream.func.composed.with.I}
{\rm
Let $\tilde{M}$ be the  orientation cover of a non-orientable surface $M = \tilde{M}/I$. Then the Lie algebra 
of $I$-invariant  exact vector fields is isomorphic to the Lie algebra of $I$-anti-invariant zero-mean  stream functions on $\tilde{M}$.
}
\end{remark}

\begin{lemma}\label{lem:stream.func.composed.with.I}
Let $\tilde{M}$ be the double orientation covering of a non-orientable surface $M = \tilde{M}/I$. 
The Lie algebra $\SoVect_I(\tilde{M} )$  of exact vector fields $v$ on $\tilde{M}$ satisfying $I_*v = v\circ I$ is isomorphic to the subalgebra of zero-mean stream functions $f$ on $\tilde{M}$ satisfying $f\circ I = -f$. 
\end{lemma}

\begin{proof}
Let $f$ be the stream function associated to a vector field $v_f \in \SoVect_I(\tilde{M})$, i.e. $df(\cdot)=\omega(v_f, \cdot)$.  
Since $v$ is $I$-invariant, while $\omega$ is $I$-anti-invariant ($I_*v = v\circ I$, $I^*\omega=-\omega$), 
the differential $df$ is $I$-anti-invariant,  $I^*df = -df$. This proves that $f$ is $I$-anti-invariant up to an additive constant,
$f\circ I=-f+C$. The zero-mean condition implies that $C=0$ and the statement follows,
 $f\circ I = -f$.
\end{proof}

Below we will denote by $\mf g_{1}^{\tilde{M} } $ and $\mf g_{-1}^{\tilde{M} } $ 
the Lie subalgebras in   $\SoVect_I(\tilde{M} )$ spanned, respectively, by $I$-invariant and $I$-anti-invariant zero-mean  stream functions. 
According to Lemma~\ref{lem:stream.func.composed.with.I} the Lie algebra of exact vector fields on $M = \tilde{M}/I$ is
$\SVect({M} )\simeq \mf g_{-1}^{\tilde{M} } $. 

\section{Curvatures of SDiff$(\mathbb K)$}\label{sec:klein.bottle.case}
\subsection{Riemannian structure on SDiff$(\mathbb K)$}

Consider the Klein bottle $\mb K$ and its orientation double covering $\mb T$, see Example~\ref{def:klein.bottle}.
We fix the standard flat metric  $g = dx_1^2 + dx_2^2 $ on~$\mb T$. 
Restrict $\SDiff(\mb T)$ to Hamiltonian diffeomorphisms $\SoDiff (\mb T )$, i.e. 
measure-preserving diffeomorphisms isotopic to the identity and leaving the ``centre of mass'' of the torus fixed. The subgroup $\SoDiff(\mb T)$ is a totally geodesic submanifold of the group $\SDiff(\mb T)$, see~\cite{arnold_topological_2021}. Its Lie algebra is $\SoVect(\mb T)$ from~\eqref{eq:stream.func}. 

Denote by $\SDiff(\mb K)$ the subgroup of $I$-invariant diffeomorphisms $\SoDiff_I(\mb T)$ (via the isomorphism of Lemma~\ref{lem:lie.group.iso}):
\[
\SDiff(\mb K) \simeq \SoDiff_I(\mb T) = 
\left\{ 
\tilde{\Phi}\in \SoDiff(\mb T)\ \vert\ \tilde{\Phi}\circ I = I\circ \tilde{\Phi}   
\right\},
\]
and its Lie algebra $\SVect(\mb K) \simeq \SoVect_I(\mb T)=\{v\in \SoVect(\mb T)~|~I_*v = v\circ I\}$.

Let the vector fields $v_f, w_g$ on $\mb T$ be associated to the stream functions $f,g \in \mf g^{\mb T}$. The Riemannian metric at the identity of the group is given by
\begin{equation*}\label{eq:riemannian.metric.stream.func}
\langle v_f, w_g \rangle  = \int_{\mb T} \left( \frac{\p f}{\p x_1} \frac{\p g}{\p x_1} + \frac{\p f}{\p x_2} \frac{\p g}{\p x_2}\right) \, dx_1dx_2 \eqcolon \langle f, g \rangle_{\mf g^{\mb T}}.
\end{equation*}
To describe below the curvatures for $\SDiff(\mb K)$ we first recall Arnold's results for the torus, and following~\cite{arnold_sur_1966} we consider the complexified Lie algebra $\mf g^{\mb T}_{\mb C}$ of $\mf g^{\mb T}$, together with a $\mb C$-linearly extended inner product $\langle.\,,. \rangle$ on $\mf g^{\mb T}_{\mb C}$.

\begin{theorem}[\cites{arnold_sur_1966, arnold_mathematical_1989}]\label{th:arnold.khesin}
Consider the Lie algebra $$\mf g^{\mb T}_{\mb C}= \{ f: \mb T \to \mb R \mid \int_{\mb T} f \, dx_1dx_2 = 0 \} \otimes \mathbb{C}$$ with the basis
\[  \{e_k \coloneq e^{i(k_1x_1 + k_2x_2)} \mid k_1, k_2 \in \mathbb{Z}\}, \]
and let $S_{\mathbb{T}}$ denote the area of the torus.
Then
\begin{small}
\begin{enumerate}
\item $\displaystyle\langle e_k, e_l \rangle= 0,$ for $k+l\neq 0$,\\
\item $\displaystyle\langle e_k, e_{-k} \rangle = \norm{k}^2 S_{\mathbb{T}} = (k_1^2+k_2^2) S_{\mb T} $,\\
\item $ \displaystyle\left\{e_k, e_l \right\} =(k\times l )e_{k+l} =(k_1l_2 - k_2l_1)e_{k+l}$,\\
\item $\displaystyle B(e_k, e_l) = (k\times l) \frac{\norm{k}^2}{\norm{k+l}^2} e_{k+l}$,\\
\item $\displaystyle \nabla_{e_k}e_l = \frac{(k_1l_2 - k_2l_1) (k_1l_1 + k_2l_2 + \norm{l}^2) }{\norm{k+l}^2} e_{k+l}$,\\
\item $\displaystyle\RCT_{k,l,m,n} = \langle \RCT(e_k, e_l)e_m, e_n \rangle = 0$ if $k+l+m+n \neq 0$,\\
\item $\displaystyle\RCT_{k,l,m,n} = (d_{ln}d_{km} - d_{lm}d_{kn}) S_{\mathbb{T}}$ if $k+l+m+n = 0$,\\
where $\displaystyle d_{uv} = \frac{(u\times v)^2}{\norm{u+v}}$
\end{enumerate}
\end{small}
\end{theorem}

\begin{corollary}\label{cor:basis.g}
A basis for the real Lie algebra $\mf{g}^{\mb T}$ is given by
\begin{equation}\label{eq:Rbasis} \left\{ \cos(k_1x_1 + k_2x_2) = \frac{1}{2}(e_k + e_{-k}), \quad \sin(k_1x_1 + k_2x_2) = \frac{1}{2i}(e_k - e_{-k}) \right\}, 
\end{equation}
with $(k_1, k_2) \in \big(\mathbb{Z}^2\setminus \{(0,0)\} \big)/\{-1, 1\}$, where $\mathbb{Z}^2/\{-1, 1\} $ is the quotient of $\mb Z^2$ under the equivalence $(k_1,k_2)\sim (-k_1,-k_2)$.
%
Equivalently the basis~\eqref{eq:Rbasis} is given by
\begin{gather}\label{eq:Rbasis1}
\big\{  \cos(k_1x_1)\cos(k_2x_2) \ \ k\in \mb N_0^2\setminus (0,0), \quad  \cos(k_1x_1)\sin(k_2x_2) \ \ k \in \mb N_0\times \mb N,
\\  \sin(k_1x_1)\sin(k_2x_2)  \ \ k\in \mb N^2, \quad  \sin(k_1x_1)\cos(k_2x_2)  \ \ k\in \mb N\times \mb N_0 \big \},\nonumber
\end{gather}
where $\mb N_0 = \{ 0, 1, 2, \dots \}$ and $\mb N = \{1, 2, \dots \}$.
\end{corollary}


\subsection{Basis for SVect$(\mathbb K)$}
Denote the subalgebra of zero-mean stream functions on $\mb T$ such that $f\circ I  =- f$ by $\mf g^{\mb T}_{-1}. $
By Lemma~\ref{lem:stream.func.composed.with.I}, we have the following isomorphisms of Lie algebras
\[ \SVect(\mb K) \simeq \SoVect_I(\mb T) \simeq \mf g_{-1}^{\mb T} . \] 
Note that the basis $\left\{e_k = e^{i(k_1x_1 + k_2x_2)} \right\}$ for $\mf g^{\mb T} \otimes \mb C$ from Theorem~\ref{th:arnold.khesin} does not restrict to a basis for $\mf g^{\mb T}_{-1}\otimes \mb C$ since $e_k\notin \mf g^{\mb T}_{-1}\otimes \mb C$ for arbitrary $k \in \mb Z^2$.  However, a part of the basis in~\eqref{eq:Rbasis1} are elements of $\mf g^{\mb T}_{-1}$. Define index sets $J^{\Re}, J^{\Im} \subset \mathbb{N}_0^2$ to remove zero functions in the basis for $\mf g^{\mb T}_{-1}$\,:
\begin{equation}\label{eq:J}
\begin{gathered}
J^{\Re} = \left\{ k = (k_1, k_2) \in \mathbb{N}_0^2 \mid  k \neq (n, 0) \text{ and }  k \neq (0, 2n) \text{ for all }n\in \mathbb{N}_0 \right\}\\
J^{\Im} = \left\{ k = (k_1, k_2) \in \mathbb{N}_0^2 \mid k \neq (0, 2n) \text{ for all }n\in \mathbb{N}_0 \right\}.
\end{gathered}
\end{equation}


\begin{theorem}\label{th:basis.g-1}
An orthogonal basis for the real Lie algebra $\SVect(\mb K) \simeq \mf g^{\mb T}_{-1}$ is given by
\begin{alignat}{6}\label{eq:B}
\mathcal{B} = \  & \big\{ \big. \mcE^\Re_k   &  & = 4\cos(k_1x_1)\cos(k_2x_2) &  & =      &  & e_k + e_{-k} + e_{\overline{k} } + e_{-\overline{k}}, \quad  &  & k \in J^{\Re}, \  &  & k_2 \text{ odd} \big. \big\}    \\
\cup \           & \big\{ \big. \mcE^{\Im}_k &  & = 4\cos(k_1x_1)\sin(k_2x_2) &  & = -i ( &  & e_k - e_{-k} - e_{\overline{k} } + e_{-\overline{k}}), \quad &  & k \in J^{\Im}, \  &  & k_2 \text{ odd} \big. \big\}  \nonumber  \\
\cup \           & \big\{ \big. \mcE^{\Re}_k &  & = 4\sin(k_1x_1)\sin(k_2x_2) &  & =      &  & e_k + e_{-k} - e_{\overline{k} } - e_{-\overline{k}}, \quad  &  & k \in J^{\Re}, \  &  & k_2 \text{ even} \big. \big\}\nonumber   \\
\cup \           & \big\{ \big. \mcE^{\Im}_k &  & = 4\sin(k_1x_1)\cos(k_2x_2) &  & = -i ( &  & e_k - e_{-k} + e_{\overline{k} } - e_{-\overline{k}}), \quad &  & k \in J^{\Im}, \  &  & k_2 \text{ even} \big. \big\} ,\nonumber
\end{alignat}
where $k = (k_1, k_2) \in \mb Z^2$,  $-k \coloneq (-k_1, -k_2)$ and $\overline{k} \coloneq (k_1, -k_2)$. 
%
\end{theorem}
Introduce the following stream functions 
\begin{equation}\label{eq:Compbasis}
\xi_k\coloneq  a_ke_k + a_{-k} e_{-k} + a_{\overline{k}} e_{\overline{k}} + a_{-\overline{k}} e_{-\overline{k}}, \quad a_k\in\mathbb C 
\end{equation}
with
\begin{equation}\label{eq:CompCond}
a_{-k} = \overline{a_k}, \quad a_{\overline{k}} = (-1)^{k_2+1}\overline{a_k},
\end{equation}
which we will denote by $\xi_{k}$ or $\xi_{(k_1, k_2)} $ if the value of $a_k$ is irrelevant. (Note that conditions~\eqref{eq:CompCond} imply that $  a_{-\overline{k}} = (-1)^{k_2+1} a_k$.)

The notation~\eqref{eq:Compbasis} compresses the information about whether $a_k$ is real, pure imaginary or complex and $k_2$ being odd or even. For instance, for an odd value of $k_2$, see also Figure~\ref{fig:vf}:
\begin{itemize}
\item if $a_k = \alpha$ is real (hence the superscript $\Re$ in $\mcE^{\Re}_k$) and $k_2$ is odd, then \\
$\xi_k=  \alpha\big(e_k + e_{-k} +e_{\overline{k}} + e_{-\overline{k}}\big) = \alpha \mcE^{\Re}_k,$

\item if $a_k = i\beta$ is pure imaginary (hence the superscript $\Im$ in $\mcE^{\Im}_k)$ and $k_2$ is odd, then 
$
\xi_k=   i \beta \big(e_k - e_{-k} - e_{\overline{k} } + e_{-\overline{k}}\big) = -\beta \mcE^{\Im}_k
$

\item if $a_k=\alpha+i\beta$ and $k_2$ is odd, then
$$
\xi_k
= \alpha \big(e_k + e_{-k} +e_{\overline{k}} + e_{-\overline{k}}\big)+ i \beta \big(e_k - e_{-k} - e_{\overline{k} } + e_{-\overline{k}}\big) = \alpha \mcE^{\Re}_k -\beta \mcE^{\Im}_k,
$$
and hence $\xi_k$ is always a real-valued function. 
\end{itemize}	
Analogously one calculates the expressions for an even value $k_2$. 

\begin{proof}[Proof of Theorem~\ref{th:basis.g-1}]
Consider the stream function 
\[ f = e_k + e_{-k} + e_{\overline{k} } + e_{-\overline{k}} = 4\cos(k_1x_1)\cos(k_2x_2) \] from~\eqref{eq:Rbasis1}. We have
\[f \circ I = 4\cos(k_1(2\pi - x_1))\cos(k_2(\pi + x_2)) = 4\cos(k_1x_1)\cos(k_2x_2 + \pi k_2) . \]
The latter equals $-f$ if and only if $k_2$ is odd and it equals $f$ if and only if $k_2$ is even. Similar statements hold for the other elements in~\eqref{eq:Rbasis1}. The index sets are defined to remove zero elements from the basis. Since $\mf g^{\mb T} = \mf g^{\mb T}_{-1} \oplus \mf g^{\mb T}_1$ the even/odd restriction on $k_2$  implies that $\mathcal{B}\subset \mf g^{\mb T}_{-1}$.

Orthogonality is checked by employing the notation \eqref{eq:Compbasis} and expanding the inner product of basis elements $\xi_k, $ $\eta_l$ using $(1)$ and $(2)$ in Theorem~\ref{th:arnold.khesin}. The only non-trivial verification is that $ \langle \xi_k ,\, \eta_k \rangle = 0$  for  $\xi_k = \mcE^{\Re}_k$ and $\eta_k = \mcE^{\Im}_k$. In that case the coefficients of $\xi_k$ and $\eta_k$ are $a_k = 1$ and  $b_k = -i$, respectively, and hence 
\[  \langle \xi_k ,\, \eta_k \rangle = \langle e_k,\, i e_{-k} \rangle +\langle e_{-k},\, -i e_{k} \rangle + (-1)^{2k_2+2}\Big(\langle e_{\overline{k}} ,\,  -i e_{-\overline{k}}\rangle +\langle e_{-\overline{k}} ,\, i e_{\overline{k}}\rangle  \Big) = 0. \]

The closedness of the Poisson bracket is also verified by a direct computation in the basis  \eqref{eq:Compbasis}. For 
two elements 
\[\xi_k=  a_ke_k + a_{-k} e_{-k} + a_{\overline{k}} e_{\overline{k}} + a_{-\overline{k}} e_{-\overline{k}}, \quad a_k\in\mathbb C, \]
\[ \eta_l=  b_le_l + b_{-l} e_{-l} + b_{\overline{l}} e_{\overline{l}} + b_{-\overline{l}} e_{-\overline{l}}, \quad b_l\in\mathbb C .\]
we obtain 
\[\{ \xi_k, \eta_l \}  = \sum_{j=\{l, -l, \overline{l}, -\overline{l}\}} (k\times j) \zeta_{k+j},\]
where $\zeta_{k+j}$ also has the form \eqref{eq:Compbasis} with coefficients $ c_{k+j} = a_kb_j,\,c_{-(k+j)} = a_{-k}b_{-j}, \,c_{\overline{k+j}} = a_{\overline{k}}b_{\overline{j}}, $ and $ c_{-(\overline{k+j})} = a_{-\overline{k}} b_{-\overline{j}} $.
\end{proof}

\begin{corollary}\label{cor:general.element.basis.klein}
A general element $\eta\in \SVect(\mb K) \simeq \mf g^{\mb T}_{-1}$ can be written as
$$\eta = \displaystyle \sum_{l \in \mathbb{Z}^2} b_le_l =  \sum_{l \in \mathbb{Z}^2} b_le^{i(k_1x_1 + k_2x_2)} , $$ with  constraints
$$		
b_{-l}\, e_{-l}= \overline{b_l} \, e_{-l},                           \\
\quad 
b_{\overline{l}} \, e_{\overline{l}}  = (-1)^{l_2+1} \, \overline{b_l} \, e_{\overline{l}},
$$
(and their combination $b_{-\overline{l}} \, e_{-\overline{l}} = (-1)^{l_2+1}\, b_l \, e_{-\overline{l}}.$)
In particular, it is enough to specify the elements $b_le_l$ for $l\in \mathbb{N}_0^2 \subset \mathbb{Z}^2$, i.e. in the first quadrant. (Here $\Re (b_l) = 0$ for $l\in \mb N_0^2 \setminus  J^{\Re}$ and $\Im(b_l) = 0 $ for $l \in \mb N_0^2 \setminus  J^{\Im}$.)
\end{corollary}

\begin{corollary}
The group $\SDiff(\mb K)$ with the Lie algebra $\SVect(\mb K) \simeq  \mf g_{-1}^{\mb T} $ 
is a totally geodesic submanifold in the group $\SoDiff(\mb T)$. In particular, its curvatures in two-dimensional planes 
containing basis elements  \eqref{eq:Compbasis} 
can be computed  by using the curvatures in the ambient group $\SoDiff(\mb T)$.
\end{corollary}

Indeed, it follows from Proposition \ref{prop:isometry}. It can also be seen directly, as  the covariant derivative of $\eta_l$ along $\xi_k$ will be given by a linear combination $ \nabla_{\xi_k}\eta_l =  \sum_{j=\{l, -l, \overline{l}, -\overline{l}\}} \Gamma_{k, j} \zeta_{k+j},$ with $\Gamma_{k,j} \in \mb R$ and $\zeta_{k+j}\in \mathcal B$, cf. (5) in Theorem \ref{th:arnold.khesin}

\begin{figure}
\includegraphics[width=0.7\linewidth]{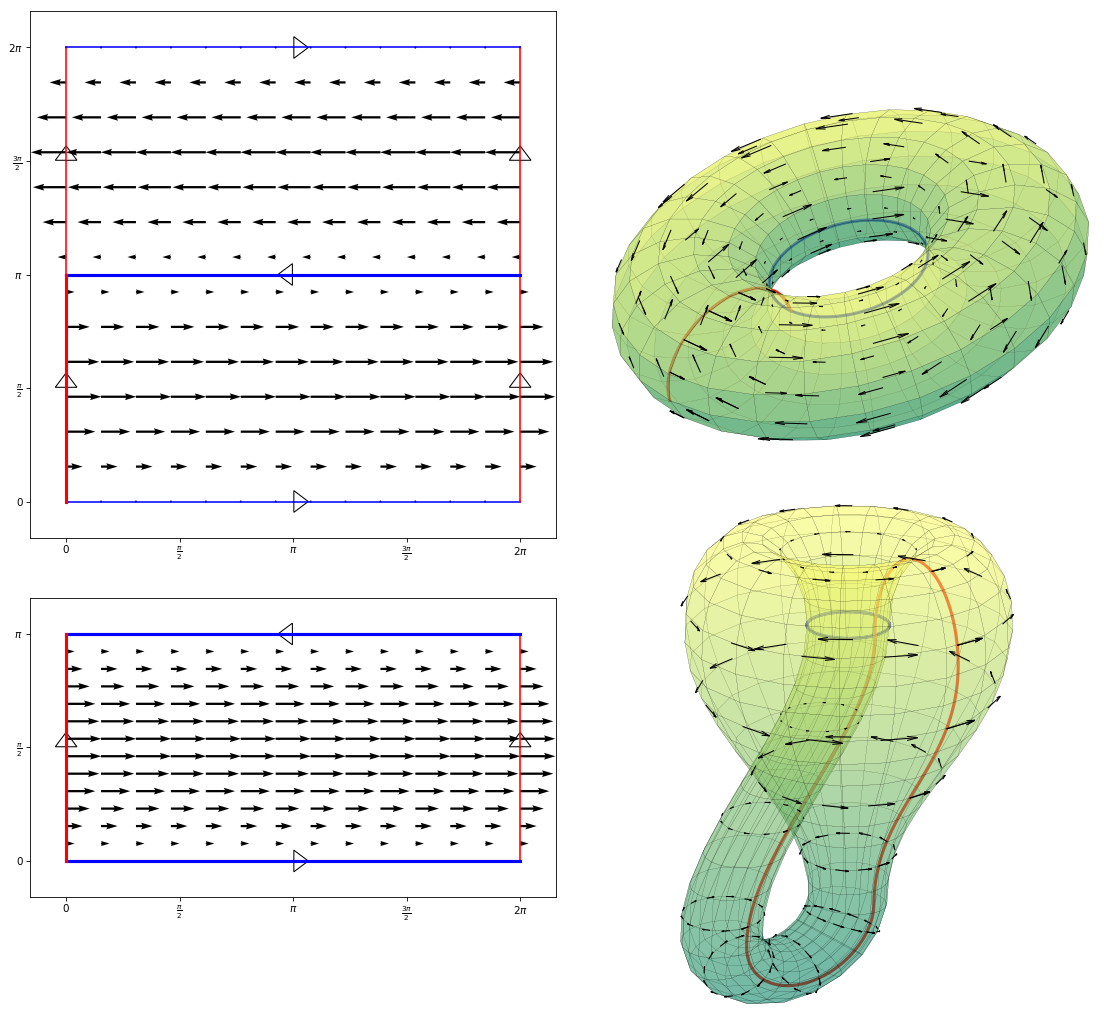}
\caption{ \tiny A vector field $v = \sin(x_2)\p_{x_1} \leftrightarrow \xi_{(0,1)} =-\cos(x_2)  $ on the torus in the subalgebra $\SoVect_I(\mb T)$, i.e. where $I_*v = v\circ I$. Thus this gives a vector field on the Klein bottle}
\label{fig:vf}
\end{figure}


\subsection{Sectional curvatures of SDiff$(\mathbb{K})$}

We start by recalling the sectional curvatures of the group $\SoDiff(\mathbb{T})$ found in~\cite{arnold_sur_1966}. Namely, in any two dimensional planes containing elements $\zeta_k$ from the basis of Corollary \ref{cor:basis.g} they are as follows.
\begin{theorem}[\cites{arnold_sur_1966,arnold_mathematical_1989}]\label{th:sec.curv.torus}
Let $$ \zeta_k =\cos(k_1x_1 + k_2x_2)= \frac{1}{2}(e_k+e_{-k}), $$
and let $ \displaystyle \eta = \sum_{l\in \mathbb{Z}^2} b_le_l$ be a general element in $\SoVect(\mb T) \simeq \mf g^{\mb T} $ with the constraint $b_{-l} = \overline{b}_l$ and orthogonal to $\zeta_k$.
Then the sectional curvature of the group $\SoDiff(\mb T)$ in any two-dimensional plane containing the stream function $\zeta_k$ is
\begin{equation}\label{eq:sec.curv.torus}
C(\zeta_k,  \eta) \cdot \left(\sum_{l\in \mb Z^2} \norm{b_l}^2\norm{l}^2 \right)= -\frac{1}{2S_{\mb T} \norm{k}^2 } \sum_{l \in \mathbb{Z}^2} \frac{(k\times l )^4}{\norm{k+l}^2} \left|b_l + b_{2k+l}\right|^2,
\end{equation}
where $S_{\mathbb{T}}$ denote the area of the torus.
\end{theorem}

The following theorem gives the sectional curvature of $\SDiff(\mb K)$ for the Klein bottle, whenever  one fixes an arbitrary stream function of the form~\eqref{eq:Compbasis}.

\begin{theorem}\label{th:sec.curv.klein}
Let $\xi_k$ be of the form~\eqref{eq:Compbasis} and 
$  \eta = \sum_{l\in \mathbb{Z}^2} b_le_l\in \mf g^{\mb T}_{-1} $ as in Corollary~\ref{cor:general.element.basis.klein}. Assume that the corresponding vector fields are orthogonal in $\SVect(\mb K)$. 
Then the sectional curvature of $\SDiff(\mathbb{K})$ in any two-dimensional plane containing $\xi_k$ is
\begin{equation}\label{eq:sec.curv.klein}
\resizebox{0.92\hsize}{!}{$%
\begin{aligned}
C (\xi_k, \eta) &
\cdot   \left( |a_k|^2 \sum_{l \in \mathbb{Z}^2} |b_l|^2 \norm{l}^2  \right)
= -\frac{1}{2S_{\mathbb{T} } \norm{k}^2} \sum_{l \in \mathbb{Z}^2} \frac{(k\times l)^4  }{\norm{k+l}^2} \left|a_k  b_l + \overline{a_k} b_{2k+l}\right|^2  \\
&+(-1)^{k_2} \frac{1}{ S_{\mathbb{T} } \norm{k}^2} \sum_{l \in \mathbb{Z}^2} \frac{ (\overline{k} \times (k+l) )^2 (k\times l)^2 }{\norm{k+l}^2} \re \left( a_k ^2 b_{l}b_{ -k+\overline{k} -l }  +  |a_k|^2 b_l b_{ -k-\overline{k} -l }    \right)\\
&+ (-1)^{k_2+1} \frac{4k_1^2k_2^2}{S_{\mathbb{T} } \norm{k}^2} \sum_{l \in \mathbb{Z}^2} l_1^2\re( a_k^2b_{l} b_{ -k + \overline{k} -l } ) + l_2^2 \re( |a_k|^2 b_l b_{-k -\overline{k} -l})\,,
\end{aligned}
$}%
\end{equation}
where $\re$ stands for taking the real part.
\end{theorem}

The proof is a direct but involved computation using expressions for sectional curvatures from Section~\ref{sec:curvature.def} and formula~\eqref{eq:Rie.curv.end.def} for the Riemannian curvature tensor. 

\begin{example}\label{ex:sec.curv.klein}
{\rm
Let us consider a pair of vectors
\[ \xi_{(1,1)} =  e_{(1,1)} + e_{(-1,-1)} +  e_{(1,-1)} + e_{(-1,1)} = 4\cos(x_1)\cos(x_2),\]
and
$$\eta_{(2,2)} = b_{(2,2)}e_{(2,2)}  +b_{(-2,-2)}e_{(-2,-2)} +b_{(2,-2)}e_{(2,-2)} +b_{(-2,2)}e_{(-2,2)}, $$
where
$$b_{(-2,-2)} = \overline{b_{(2,2)}} , \quad b_{(2,-2)} = -\overline{b_{(2,2)}}, \quad b_{(-2,2)} = - b_{(2,2)},  $$
and compute the sectional curvature $C(\xi_{(1,1)}, \eta_{(2,2)})$ using Theorem~\ref{th:sec.curv.klein}. The terms $b_l$ in the sums in~\eqref{eq:sec.curv.klein} are non-zero for
$$l\in J= \{(2,2), (-2,-2), (2,-2), (-2,2) \}. $$  
Notice that here $k = (1,1)$ and for all $l\in J$,
$$ 2k+l \notin J, \quad -k-\overline{k} - l \notin J, \quad -k+\overline{k} - l \notin J .$$
Thus for all $l \in J$,
$b_{2k+l} = b_{-k-\overline{k} - l} = b_{-k+\overline{k} - l} = 0$. This simplifies the calculation of $C(\xi_{(1,1)}, \eta_{(2,2)})$ to
\[ C(\xi_{(1,1)}, \eta_{(2,2)})  \left(  32 |b_{(2,2)}|^2  \right)= -\frac{1}{4S_{\mathbb{T} } } \sum_{l \in J \cup (J-2k)} \frac{(k\times l)^4  }{\norm{k+l}^2} \left| b_l \right|^2.  \]
This implies that  $C(\xi_{(1,1)}, \eta_{(2,2)})  \simeq -0.0203$. 

}
\end{example}

\begin{remark}\label{rem:basis.comparison}
{\rm It is worth mentioning that the basis~\eqref{eq:B} for $\mf g^{\mb T}_{-1} $ for the Klein bottle
does not ``behave" as nicely as the basis
\[ 
\big\{ 2\cos(k_1x_1 + k_2x_2) =  e_k + e_{-k}, \quad 2\sin(k_1x_1 + k_2x_2) = -i(e_k - e_{-k})  \mid k \in \mb Z^2/ \{-1,1 \} \big\} 
\]
 for the torus in Theorem~\ref{th:sec.curv.torus}. For the torus the sectional curvature turns out to be negative in any two-dimensional plane containing the direction $e_k + e_{-k}$, but it is not the case for the Klein bottle. In a nutshell, the reason  
 is that basis elements for $\mathbb{K} $ are sums of four parts (rather than two for $\mathbb{T}$),  the corresponding  curvature tensor $\RCT_{k,l,m,n}$ contains $4^4 = 256$ terms and there are more interactions between them than in the torus case, resulting in two extra sums in~\eqref{eq:sec.curv.klein} not appearing in~\eqref{eq:sec.curv.torus}.
}\end{remark}

The following corollary shows that if two stream functions of the form~\eqref{eq:Compbasis} are \textit{sufficiently different}, the sectional curvature in the plane through them is strictly negative.

\begin{corollary}\label{cor:neg.C.basis}
Let  $k_1\neq l_1$, $k_2\neq l_2$ for two stream functions $\xi_k$, $\eta_l$ as in~\eqref{eq:Compbasis}.
Then the sectional curvature of the group $\SDiff(\mathbb{K})$ in the two-dimensional direction given by $\xi_k$ and $\eta_l$ is
\begin{equation*}
C(\xi_k, \eta_l) = -\frac{1}{4 S_\mathbb{T} \norm{k}^2 \norm{l}^2 }\left(\frac{(k\times l)^4}{ \norm{k+l}^2 } + \frac{(k\times l)^4}{ \norm{k-l}^2 } + \frac{(k\times \overline{l} )^4}{ \norm{k+\overline{l}}^2 }  + \frac{(k\times \overline{l})^4}{ \norm{k-\overline{l}}^2 }  \right) ,
\end{equation*}
and, in particular, it is strictly negative and bounded as follows:
\[ -\frac{ \min\{\norm{k}^2, \norm{l}^2\} }{S_{\mb T}} \leq C(\xi_k, \eta_l)  <0 . \]
\end{corollary}
\begin{proof}
Having fixed index $l$ for $\eta_l$,  we use $j$ as the summing index in~\eqref{eq:sec.curv.klein}. The two last sums in~\eqref{eq:sec.curv.klein} vanish as all products $b_jb_{-k+\overline{k} -j}$ and $b_j b_{-k-\overline{k} - j}$ are zero. In the first sum, $b_j$ is non-zero when $j\in \{l, -l, \overline{l}, -\overline{l}\},$ and similarly with $b_{2k+j}$. 
Employing the normalization from the l.h.s. of~\eqref{eq:sec.curv.klein}, we get the desired result.
\end{proof}

Recall that by Theorem~\ref{th:sec.curv.torus} the sectional curvatures of $\SoVect(\mb T)$ are always negative in planes containing  $\zeta_k$ from the basis of Corollary~\ref{cor:basis.g}. However, the following theorem shows that in $\SVect(\mb K)$  in certain two-dimensional directions containing a stream function $\xi_k$, which is an analogue of $\zeta_k$ for $\mathbb T$, the sectional curvature is positive.

\begin{theorem}\label{th:pos.sequence}
Let $\xi_k=  a_ke_k + a_{-k} e_{-k} + a_{\overline{k}} e_{\overline{k}} + a_{-\overline{k}} e_{-\overline{k}},$ with $a_k\in \mb R$ or $i \mb R$, as in~\eqref{eq:Compbasis}.
Then there exists a direction $\eta \in \SVect(\mb K)$ such that $C(\xi_k, \eta) >0$. 

Namely, there exists a sequence of stream functions $\eta_{l_m}\in \SVect(\mb K)$ 
such that $\{C(\xi_k, \eta_{l_m})\}_{m\in \mathbb{N}} $ is a monotonically increasing sequence, converging to
\[ \lim\limits_{m\to \infty} C(\xi_k, \eta_{l_m}) = \begin{cases}
\displaystyle\frac{k_2^2}{2\norm{k}^2 S_{\mb T}} (4k_1^2 -3 k_2^2)  ,  & \text{if } k_1\geq k_2   \\[7pt]
\displaystyle  \frac{k_1^2}{2\norm{k}^2 S_{\mb T}} (4k_2^2 -3 k_1^2) , & \text{if } k_1\leq k_2 .
\end{cases}\]

\end{theorem}

\begin{proof}
Given $k$ with $k_1 \geq k_2$  we set
$$  \eta_{l_m}  =b_{(m, k_2)} e_{(m, k_2)} + b_{(-m, -k_2)} e_{(-m, -k_2)}+ b_{(m, -k_2)}e_{(m, -k_2)} + b_{(-m, k_2)}e_{(-m, k_2)}  \in \mathcal{B},$$ for $m\neq k_1$. We will later specify $b_{(m,k_2)} $ to be $ 1$ or $-i$, but either way we will choose $|b_l| = 1$.
We now compute $C(\xi_k, \eta_{l_m})$ by using Theorem~\ref{th:sec.curv.klein}.

The normalization constant on the left-hand side of~\eqref{eq:sec.curv.klein} equals
$4|a_k|^2 \norm{l}^2=4|a_k|^2(m^2 + k_2^2).$
By expanding the first sum in~\eqref{eq:sec.curv.klein} in terms of $m =  j_1$, we obtain
\[  -\frac{|a_k|^2}{2S_{\mb T} \norm{k}^2} \sum_{j\in \mb Z^2 } \frac{(k_1k_2 - j_1k_2)^4}{(k_1 + j_1)^2 + (k_2 + j_2)^2}  =  -4|a_k|^2\frac{k_2^4}{S_{\mb T} \norm{k}^2} m^2 + O\left(\frac{1}{m}\right),\]
as $m\to\infty$.
The expansion of the second sum in~\eqref{eq:sec.curv.klein} in terms of $m =  j_1$ gives
\[  (-1)^{k_2} \frac{1}{ S_{\mathbb{T} } \norm{k}^2} \sum_{j\in \mb Z^2 }\frac{ ( k_1k_2 + j_1k_2  )^2 (k_1k_2 - j_1k_2)^2 }{(k_1 + j_1)^2 } \re \left( a_k ^2 (-1)^{k_2+1} b_{j}^2  \right)   \]
\[ = -2a_k^2b_l^2 \frac{k_2^4}{ S_{\mathbb{T} } \norm{k}^2} m^2 + O\left(\frac{1}{m}\right).  \]

The third sum in~\eqref{eq:sec.curv.klein} simplifies to 
\[   (-1)^{k_2+1} \frac{4k_1^2k_2^2}{S_{\mathbb{T} } \norm{k}^2} \sum_{j\in \mb Z^2 } j_1^2\re( a_k^2 (-1)^{k_2+1} b_{j}^2 )   =  8a_k^2b_l^2 \frac{ k_1^2k_2^2 }{S_{\mathbb{T} } \norm{k}^2} m^2 .\]
Combining these  terms and dividing by the normalization constant, we obtain
\[ C(\xi_k, \eta_{l_m}) = \frac{k_2^2}{2|a_k|^2\norm{k}^2  S_{\mb T}}  \left(4a_k^2b_l^2k_1^2  -a_k^2b_l^2k_2^2 - 2|a_k|^2k_2^2\right)  + O\left(\frac{1}{m}\right) . \]
Choose $b_l$ such that $a_k^2b_l^2 >0$, i.e. set $b_l = 1$ for $a_k\in \mb R$ and set $b_l = -i$ for $a_k \in i\mb R$. Then
\[ C(\xi_k , \eta_{l_m}) = \frac{k_2^2}{2\norm{k}^2  S_{\mb T}}   \left(4k_1^2 - 3k_2^2\right)  + O\left(\frac{1}{m}\right) , \]
which is positive for large $m$, since $k_1\geq k_2$. 

Finally, a vector $\eta$ satisfying $C(\xi_k, \eta) > 0$ can be chosen as $\eta=\eta_{l_m}$ for $m$ large enough, so that $C(\xi_k, \eta_{l_m}) > 0$. 
\end{proof}

\begin{example}\label{ex:PosNeg}
{\rm Table~\ref{tab:pos.seq} illustrates Theorem~\ref{th:pos.sequence} for
\begin{equation}\label{eq:xi(20,10).ex}
\xi_{(20, 10)} =e_{(20,10)} + e_{(-20, -10)} - e_{(20,-10)} - e_{(-20, 10)}.
\end{equation}
The proof of Theorem~\ref{th:pos.sequence} suggests that for
$$ \eta_{(m,10)} = e_{(m,10)}+ e_{(-m,-10)} -  e_{(m,-10)} - e_{(-m,10)} ,$$
with $m\in \mb N$  the sequence $C(\xi_{(20,10)} , \eta_{(m, 10)})$ converges to $\frac{65}{2\pi^2} \simeq 3.2929$.

\begin{small}
\begin{table}
\centering
\begin{tabular}{|c|c|c|c|c|c|c|c|c|c|c|c|c|c|}
\hline
$m$ & 0 & 5 & 10 & 15 & 20 & 25  & 35  & 45 & 100 & 300 & 500  \\\hline
$C(\xi_{(20,10)}, \eta_{(m,10)})$ & 0 & -1.6 & -0.3 & 0.5 & 0 & 1.1  & 1.7  & 2.2 & 3.1 & 3.2 & 3.2843  \\\hline
\end{tabular}
\bigskip
\caption{ \tiny The table shows the sectional curvatures of $\SDiff(\mathbb{K})$ in the two dimensional planes spanned by $\xi_{(20,10)} $ and $\eta_{(m, 10)}$ for selected values of $m\in \mb N$.  The sequence of stream functions $\eta_{(m,10)}$ is given by Theorem~\ref{th:pos.sequence}. As $m\to\infty $,  $C(\xi_{(20,10)}, \eta_{(m, 10)})$ will converge to $\frac{65}{2\pi^2} \simeq 3.2929$. The graph of the sequence is shown in Figure~\ref{fig:blanket}. }
\label{tab:pos.seq}
\end{table}
\end{small}

Note that the sequence $\eta_{(m, 10)}$ contains all stream functions $\eta_l $ of type~\eqref{eq:Compbasis} such that\\
${C(\xi_{(20,10)}, \eta_l )>0}$ since the elements $\eta_l$ with $l = (l_1, l_2)$,  $l_1 \neq k_1=20$ and $l_2 \neq k_2=10$ satisfy the hypothesis of Corollary~\ref{cor:neg.C.basis}. Candidates $\eta_l$ for which $C(\xi_{(20,10)}, \eta_l)$ is positive are on the lines $l_1 = k_1= 20$ or $l_2 = k_2 = 10$. The construction in Theorem~\ref{th:pos.sequence} implies that if the coefficients of $\eta_l$ are imaginary, then ${C(\xi_{(20,10)}, \eta_l )<0}$, and if $l_1 = k_1 = 20$, then $C(\xi_{(20,10)}, \eta_{(20, l_2)}) \leq 0$.

\begin{figure}
\centering
\includegraphics[width=0.6\textwidth]{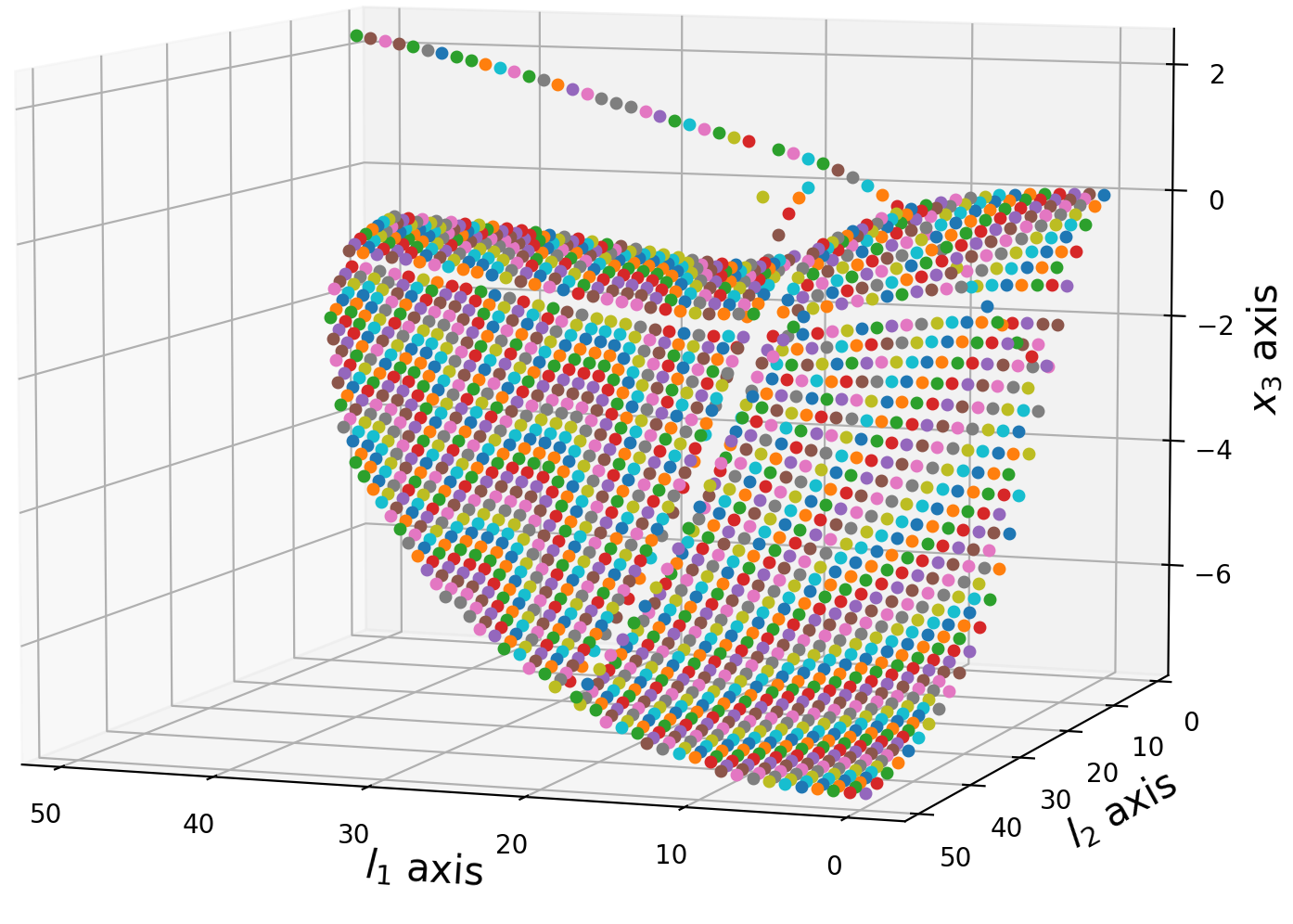}
\caption{ \tiny The value of $C(\xi_{(20,10)}, \eta_l)$ along the $x_3$-axis for $\eta_l$ with various $l = (l_1, l_2)\in J^{\Im}$ satisfying $\norm{l}\leq 50$.
On the ``cross" passing through $l_1 = 20$, $l_2 = 10$,  the hypothesis of Corollary~\ref{cor:neg.C.basis} is not fulfilled. The sequence $\eta_{(m, 10)}$ from Table~\ref{tab:pos.seq} is going along $l_2  = 10$. (The colours are to distinguish individual points.)}
\label{fig:blanket}
\end{figure}

Figure~\ref{fig:blanket} shows the graph of $C(\xi_{(20,10)},  \eta_l )$ as a function of $l = (l_1, l_2) $, where the sequence $\eta_{(m,10)}$ is clearly visible. The points on the graph of $C(\xi_{(20,10)}, \eta_l)$ happen to lay on a smooth surface, with the exception of the elements $\eta_l $ with $l_1 = k_1 $ or $l_2 = k_2$. This observation is made rigorous in the next section where the asymptotic value of $C(\xi_k, \eta_l)$ is calculated as $\norm{l} \to \infty$.
}\end{example}

\subsection{Asymptotics of sectional curvatures for SDiff$(\mathbb{K})$}
In order to compute the normalized Ricci curvature in $\SoDiff_I(\mb T)$, we first describe the asymptotic behaviour of  sectional curvatures as the eigenvalues of the Laplace-Beltrami operator tends to infinity, see Figure~\ref{fig:blanket}.
%
%
%

\begin{theorem}\label{th:sec.curv.large.l}
Fix a stream function $\xi_k$ as in~\eqref{eq:Compbasis} and let $ \displaystyle \kappa \coloneq  {\rm arctan}\left(\frac{k_2}{k_1}\right).$  Consider a sequence of stream functions $\eta_{l_m}$  where $(l_m)_1 \neq k_1$, $(l_m)_2 \neq k_2$, and
\[\norm{l_m} \stackrel{m\to \infty}{\longrightarrow} \infty, \quad \frac{l_m}{\norm{l_m}} \longrightarrow l, \quad \lambda  =  {\rm arctan}\left(\frac{l_2}{l_1}\right). \]
$$
\text{Then }\quad \lim\limits_{m\to \infty} C(\xi_k, \eta_{l_m})
= -\frac{\norm{k}^2}{2S_{\mb T}} \left(\sin^4(\lambda - \kappa) +\sin^4(\lambda + \kappa) \right)\,.
$$
\end{theorem}
\begin{proof}
Under the hypothesis of the theorem one can apply Corollary~\ref{cor:neg.C.basis} to obtain
\[ C (\xi_k, \eta_{l_m} ) = -\frac{1}{4S_{\mb T} \norm{k}^2 \norm{l_m}^2} \sum_{j_m \in J_m} \frac{(k\times j_m )^4}{\norm{k+j_m}^2}, \quad \text{for } j_m\in \{l_m, \ -l_m, \ \overline{l_m}, \ -\overline{l_m} \} = J_m  .  \]
We have that
\begin{equation}\label{eq:cross.prod.formula}
\begin{gathered}
k\times j_m= \norm{k}\norm{j_m}\sin(\theta^k_{j_m})\,,\\
k\times j_m= k\times (k+j_m)  = \norm{k}\norm{k+j_m}\sin(\theta^k_{k+j_m})\,,
\end{gathered}
\end{equation}
where $\theta^k_{j_m}$ is the angle between the vectors  $k$ and $j_m$, while 
$\theta^k_{k+j_m}$  is the angle between $k$ and $k+j_m$.
Substituting~\eqref{eq:cross.prod.formula} into the formula for $C (\xi_k, \eta_{l_m} ) $ we obtain
\[
C (\xi_k, \eta_{l_m} ) = -\frac{\norm{k}^2}{4S_{\mb T}} \sum_{j_m \in J_m}  \sin^2(\theta^k_{k+j_m})\sin^2(\theta^k_{j_m}).
\]
It is clear that
\[ \lim\limits_{m\to \infty }  \sin\left(\theta^k_{k+j_m}\right) = \lim\limits_{m\to \infty } \sin\left(\theta^k_{j_m}\right), \]
since $\lim\limits_{m\to \infty} \frac{k}{\norm{j_m}} = (0,0)$, as $\norm{j_m}\longrightarrow \infty$ by hypothesis.
Putting this together with the expression for $C(\xi_k, \eta_{l_m})$, we obtain
\begin{equation}\label{eq:limit.sectional.curvature.intermediate.step}
\lim\limits_{m\to \infty  } C(\xi_k, \eta_{l_m}) = \lim\limits_{m\to \infty } -\frac{\norm{k}^2}{4S_{\mb T}} \sum_{j_m \in J_m} \sin^4(\theta^k_{j_m}) = -\frac{\norm{k}^2}{4S_{\mb T}} \sum_{l \in J} \sin^4(\theta^k_{l}).
\end{equation}

\begin{figure}[ht]
\centering
\includegraphics[width=0.37\textwidth]{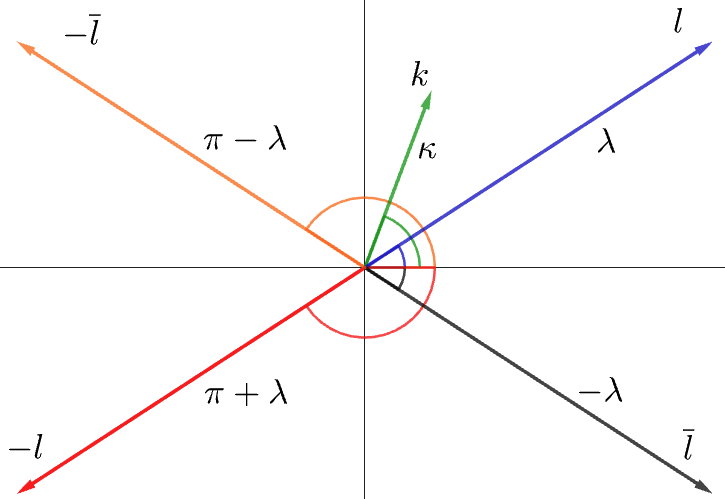}
\caption{ \tiny The vectors $k, \ l, \ -l, \ \overline{l}, \ -\overline{l}$ and corresponding angles $\kappa, \ \lambda, \  {\pi + \lambda} ,  \ -\lambda, \ \pi -\lambda $ between them and the positive $x$-axis.}
\label{fig:angles}
\end{figure}

Now express $\theta^k_{l}$ in terms of the angle $\kappa$ between $k$ and the positive $x$-axis, and the angle $\lambda$ between $l$ and the positive $x$-axis. As one can see in Figure~\ref{fig:angles},
\begin{equation*}
\theta^k_{l} = \lambda - \kappa, \quad \theta^k_{-l} = \pi + \lambda  - \kappa  \quad \theta^k_{\overline{l}} = -\lambda - \kappa, \quad \theta^k_{-\overline{l} } = \pi - \lambda - \kappa,
\end{equation*}
and thus
\begin{align*}
\sin^4(\theta^k_{l} ) &= \sin^4(\lambda - \kappa  ), & \sin^4(\theta^k_{-l}) &=  \sin^4(\lambda - \kappa), \\ 
\sin^4(\theta^k_{\overline{l}}) &= \sin^4(\lambda + \kappa),  &  \sin^4(\theta^k_{-\overline{l}}) &= \sin^4(\lambda + \kappa).
\end{align*}
Combine these with~\eqref{eq:limit.sectional.curvature.intermediate.step} to obtain
\[ \lim\limits_{m\to \infty  } C(\xi_k, \eta_{l_m}) =  -\frac{\norm{k}^2}{2S_{\mb T}} \left( \sin^4(\lambda - \kappa) + \sin^4(\lambda + \kappa) \right). \]
\end{proof}


\subsection{Normalized Ricci curvature of SDiff$(\mathbb{K})$}\label{sec:ricci.klein}

We now give a definition of the normalized Ricci curvature of $\SoDiff(M)$ for a compact Riemannian surface $M$. It is an infinite-dimensional analogue of the normalized Ricci curvature for a finite-dimensional manifold, see e.g.~\cite{arnold_topological_2021}.
\begin{definition}\label{def:norm.ric.curv}
Let $(M, g)$ be a compact Riemannian manifold without boundary and
$$\Delta^g  = \frac{1}{\sqrt{|g|}} \p_i \left(\sqrt{|g|}g^{ij} \p_j\right)$$
the Laplace-Beltrami operator on $M$.  The eigenfunctions 
\[\Delta^g \mcE_i  = \lambda_i \mcE_i \] 
can be considered as a basis for $\SoVect(M)$.  Define a (partial) order on this basis by descending eigenvalues, $\lambda_i \geq \lambda _{i+1}$.
The normalized Ricci curvature of $\SoDiff(M)$ in the direction $\xi$ is
\[ \Ric(\xi) = \lim\limits_{N\to \infty } \frac{1}{N-1}\sum_{i=1}^{N} C(\xi, \mcE_i)  \in \mathbb{R}\cup \{\pm\infty\} .\]
\end{definition}

When computing the normalized Ricci curvature we consider a subsequence of
\begin{equation}\label{eq:SN}
S_N (\xi) = \frac{1}{N-1} \sum_{i=1}^{N} C(\xi, \mcE_i)   \,,
\end{equation}  
which only employs the first $N$ elements $\mcE_i $ of the basis with $\sqrt{|\lambda_i|}  \leq R$.  

\medskip

Let now $M$ be the  torus or Klein bottle, and we consider the basis $ \mathcal{B}$ and $\xi\coloneq\xi_k\in  \mathcal{B}$.
We need to express $N$ as a function of $R$.

\begin{lemma}\label{lem:basis.element.counting}
Let $  |\mathcal{B}_R | $ be the number of basis elements $\mcE^{\Re}_l, \mcE^{\Im}_l$ from the basis $\mathcal{B}$ in Theorem~\ref{th:basis.g-1} with $\norm{l}\leq R$. The sequence is given by
\[|\mathcal{B}_R |= 2\left(A(R) - \Big\lfloor \frac{R}{2} \Big\rfloor -1 \right)  -R, \]
where $A(R)$ is the number of non-negative integer solutions to the inequality
\[x^2 +y^2 \leq R^2.\]
In particular, $|\mathcal{B}_R | =  \frac{\pi}{2}R^2 + O(R)$ as $R\to\infty$.
\end{lemma}
\begin{proof}
Counting the elements in the basis is the same as counting the elements in the index sets in~\eqref{eq:J} with $\norm{l} \leq R$:
\begin{equation*}
\begin{gathered}
J^{\Re}_R = \left\{ l = (l_1, l_2) \in \mathbb{N}_0^2 \mid l_1^2 +l_2^2 \leq R^2 , \, \,    l \neq (n, 0) \text{ and }  l \neq (0, 2n) \text{ for all }n\in \mathbb{N}_0 \right\},\\
J^{\Im}_R = \left\{ l = (l_1, l_2) \in \mathbb{N}_0^2 \mid  l_1^2 +l_2^2 \leq R^2 , \, \,    l \neq (0, 2n) \text{ for all }n\in \mathbb{N}_0 \right\}.
\end{gathered}
\end{equation*}
The number of points in $J^{\Re}_R$ is the number of integer solutions to $x^2 + y^2 \leq R^2$, minus $R$ points of the type $(n, 0)$ for $0\leq n \leq R $; $ \big\lfloor \frac{R}{2} \big\rfloor$ points of the type $(0, 2n)$ for $0\leq 2n \leq R$; and minus the origin. Hence $|J^{\Re}_R| = A(R) - \big\lfloor \frac{R}{2} \big\rfloor -1  -R. $
Similarly, $|J^{\Im}_R| = A(R) - \big\lfloor \frac{R}{2} \big\rfloor -1 . $\\

To prove the growth rate we note that there are $(R+1)^2$ points $(l_1, l_2) \in \mb N_0^2$ in the box $[0, R]\times [0,R] \subset \mb R^2$. $A(R)$ is defined by intersecting this box with the disk $x^2 + y^2 \leq R^2$, hence $|A(R)|$ grows as $\frac{\pi}{4} R^2 + O(R)$ and thus $|\mathcal B_R|$ grows as $\frac{\pi}{2}R^2 + O(R)$. 
\end{proof}

Finally we need a bound on the sectional curvature for the stream functions $\xi_k, \eta_l$ as in~\eqref{eq:Compbasis} that are not  in a sequence satisfying Theorem~\ref{th:sec.curv.large.l}.

\begin{lemma}\label{lem:bounded.sec.curv.klein}
Let $\xi_k, \eta_l$ be as in~\eqref{eq:Compbasis} with $\xi_k \neq \eta_l$.
Then
\[ \big|C(\xi_k, \eta_l)\big| \leq \frac{9\norm{k}^2}{2 S_{\mb T}} + \frac{2\min\{k_1^2, k_2^2\}}{S_{\mb T}}. \]

\end{lemma}

\begin{proof}
We apply Theorem~\ref{th:sec.curv.klein} with the two elements 
\[\xi_k = a_ke_k + a_{-k} e_{-k} +a_{\overline{k}} e_{\overline{k}} +  a_{-\overline{k}}e_{-\overline{k}},\quad |a_k| = |a_{-k} | = |a_{\overline{k}}| = |a_{-\overline{k}}| \,, \]
\[\eta_l = b_le_l + b_{-l} e_{-l} +b_{\overline{l}} e_{\overline{l}} +  b_{-\overline{l}}e_{-\overline{l}}, \quad|b_l| = |b_{-l} | = |b_{\overline{l}}| = |b_{-\overline{l}}|,  \]
and $k\neq l \in \mb N_0^2$. 
The normalization factor in~\eqref{eq:sec.curv.klein} is $4|a_k|^2|b_l|^2 \norm{l}^2$.
The first sum in~\eqref{eq:sec.curv.klein} is bounded by
\[\frac{16 \norm{k}^2}{S_{\mb T}} \norm{l}^2 |a_k|^2 |b_l|^2, \]
since there are 8 index pairs $j = (j_1, j_2) \in \mb Z^2$ such that $0<|a_kb_j + \overline{a_k} b_{2k+j}|^2 \leq 4|a_k|^2|b_j|^2,$ while all other indices do not contribute.
The second sum in~\eqref{eq:sec.curv.klein} is bounded by 
\[\frac{2 \norm{k}^2}{S_{\mb T}} \norm{l}^2 |a_k|^2 |b_l|^2, \]
since there are at most 2 index pairs $j = (j_1, j_2) \in \mb Z^2$ such that $0<\re(a_k^2 b_jb_{-k+\overline{k} -j} + |a_k|^2b_j  b_{-k-\overline{k}- j}) \leq |a_k|^2|b_j|^2,$ while there is zero contribution from all other index pairs. 

The third sum in~\eqref{eq:sec.curv.klein} is bounded by 
\[\frac{8 \min\{k_1^2, k_2^2\} }{ S_{\mb T}} \norm{l}^2 |a_k|^2 |b_l|^2, \]
since there are at most 2 index points $j = (j_1, j_2) \in \mb Z^2$ such that either $0<\re(a_k^2 b_jb_{-k+\overline{k} -j} ) \leq |a_k|^2|b_j|^2,$ or  $0<\re(|a_k|^2b_j  b_{-k-\overline{k}- j} ) \leq |a_k|^2|b_j|^2 $,  with no contribution from other indices.
\end{proof}

\begin{theorem}\label{th:ric.klein}
Let $\xi_k=  a_ke_k + a_{-k} e_{-k} + a_{\overline{k}} e_{\overline{k}} + a_{-\overline{k}} e_{-\overline{k}},$ with $a_k\in\mathbb C$, as in~\eqref{eq:Compbasis}.
Then the normalized Ricci curvature of $\SDiff(\mb K)$ in the direction $\xi_k$ is
\[ \Ric(\xi_k) = -\frac{3\norm{k}^2}{8S_{\mb T}}  = -\frac{3\norm{k}^2}{16S_{\mb K}} .\]
\end{theorem}
\begin{proof}
Consider a partial order on $\mcE_l= b_le_l + b_{-l} e_{-l} + b_{\overline{l}} e_{\overline{l}} + b_{-\overline{l}} e_{-\overline{l}} \in \mathcal{B}$ for $e_l \coloneq e^{i(l_1x_1 + l_2x_2)}$ by descending eigenvalues 
\[ \Delta\mcE_l = \left(\frac{\p^2}{\p x_1^2} + \frac{\p^2 }{\p x_2 ^2}  \right) \mcE_l  = -\norm{l}^2\mcE_l. \]
Let $S_N(\xi_k)$ be the sequence~\eqref{eq:SN} converging to the normalized Ricci curvature. 
Let  
\begin{equation}\label{eq:BR} 
\mathcal{B}_R = \left\{ \mcE_l \in \mathcal{B}  \mid \norm{l}\leq R \right\}.  
\end{equation}
From now on we consider the subsequence of $S_N(\xi_k)$ where $N$ is restricted to $N = \big|\mathcal{B}_R\big|$, where $\big|\mathcal{B}_R\big|$ is the number of basis elements in $\mathcal{B}$ with $\norm{l}\leq R$, see Lemma~\ref{lem:basis.element.counting}. Then 
\begin{equation}\label{eq:SN.BR}
S_N (\xi_k)= S_{|\mathcal{B}_R|} (\xi_k)= \frac{1}{|\mathcal{B}_R| }  \sum_{\mcE_l \in \mathcal{B}_R} C(\xi_k, \mcE_l).
\end{equation}
The subsets of $\mathcal B_R$ which grows as $\sim R^2$ will determine the value of the sequence. 
\begin{itemize}
\item For fixed $k$ the sectional curvatures $ C(\xi_k, \mcE_l) $ are uniformly bounded when varying $\mcE_l$, see Lemma~\ref{lem:bounded.sec.curv.klein}. One can thus remove a finite set $D$, i.e. a set where $|D|$ does not depend on $R$, from $\mathcal{B}_R$ without changing the value of $\lim\limits_{R\to\infty} S_{|\mathcal{B}_R|}$. 
\item Similarly, one can remove \textit{thin} sets $E_R$, where $|E_R|$ grows as $\sim R$, from $\mathcal{B}_R$, since if $C_k$ is the constant bounding $C(\xi_k, \mcE_l) $ from Lemma~\ref{lem:bounded.sec.curv.klein}. Indeed, given any $\epsilon>0$
\[ \left| \frac{1}{|\mathcal{B}_R| }  \sum_{\mcE_l \in E_R} C(\xi_k, \mcE_l) \right| \leq  \frac{1}{|\mathcal{B}_R| } |E_R| C_k \leq \frac{cR}{\frac{\pi}{2}R^2 + c_1R + c_0} C_k <\epsilon ,  \]
for $R$ sufficiently large. 
\end{itemize}

\begin{figure}
\centering
\includegraphics[width=0.3\textwidth]{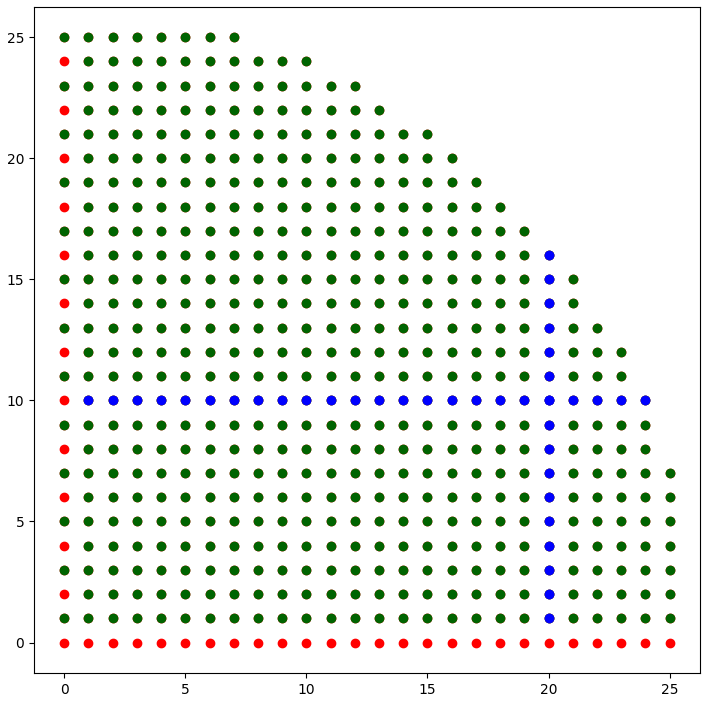}
\caption{ \tiny   Here $R=25$ and $k = (20, 10)$ as in Figure~\ref{fig:blanket}. The green points correspond to the indices $l$ such that $\mcE_l^{\Re} \in A_R $, the blue to the indices $l$ such that $\mcE_l^{\Re} \in  \mathcal{B}_R \setminus A_R$ and the red to the $l$ such that $\mcE_l^{\Re} \notin  \mathcal{B} $.   }
\label{fig:AR_complement}
\end{figure}
Let $A_{R}$ be the set of basis elements $\mcE^{\Re}_l, \mcE^{\Im}_l \in \mathcal{B}_R$ where $\norm{l}\leq R $, $k_1\neq  l_1$ and $k_2 \neq l_2$. 
The set $A_R$ grows as $\sim R^2$. Actually, since the indices of $\mathcal B_R\setminus A_R$ consist of four lines, two in each copy of $\mb N_0^2$, see Figure~\ref{fig:AR_complement}, $|\mathcal B_R\setminus A_R | \leq 4R $ and hence both
$|A_R|$ and $|\mathcal{B}_R|$ grow as $\frac{\pi}{2}R^2 $, see Lemma~\ref{lem:basis.element.counting}. By removing the \textit{thin} set $\mathcal{B}_R \setminus A_R$ we obtain that 
\begin{equation*}\label{eq:SBR.cross.removed}
\lim\limits_{R\to\infty} S_{|\mathcal{B}_R|} = \lim\limits_{R\to\infty}  \frac{1}{|\mathcal{B}_R|} \sum_{\mcE_l \in A_R} C(\xi_k, \mcE_l)  . 
\end{equation*}
Now set
\[f(\lambda) = \frac{\norm{k}^2}{8S_{\mb T}} \left(4\cos(2\lambda) \cos(2\kappa) - \cos(4\lambda) \cos(4\kappa)\right), \quad \lambda  =  {\rm arctan}\left(\frac{l_2}{l_1}\right) , \ \kappa =  {\rm arctan}\left(\frac{k_2}{k_1}\right).  \] 
By Theorem~\ref{th:sec.curv.large.l} and using trigonometric identities 
we have that for $\norm{l} = R$
\[ \lim\limits_{R\to \infty } C(\xi_k, \mcE_l) = - \frac{3\norm{k}^2}{8S_{\mb T}} + f(\lambda). \]
By removing the finite set $A_{R_1}$, where $C(\xi_k, \mcE_l)$ is not well approximated by the asymptotic function, we get 
\begin{multline*}
\lim\limits_{R\to \infty} S_{|\mathcal B_R|} = \lim\limits_{R\to\infty} \frac{1}{|\mathcal{B}_R|}  \sum_{\mcE_l \in A_R\setminus A_{R_1}} \Big(- \frac{3\norm{k}^2}{8S_{\mb T}} + f(\lambda) \Big)  \\
= - \frac{3\norm{k}^2}{8S_{\mb T}}  \lim\limits_{R\to\infty} \frac{|A_R\setminus A_{R_1}|}{|\mathcal{B}_R|}  +\lim\limits_{R\to\infty} \frac{1}{|\mathcal{B}_R|}  \sum_{\mcE_l \in A_R\setminus A_{R_1}}  f(\lambda) .
\end{multline*}
By auxiliary Lemma~\ref{lem:klein.ricci.sum.over.f}, proved in Appendix, the term with $f(\lambda)$ vanishes.
Since the set $|A_R\setminus A_{R_1}|$ grows as $\frac{\pi}{2} R^2$, we get the desired result.
\end{proof}


\section{Curvatures of SDiff$(\mathbb{RP}^2)$}\label{sec:real.projective.plane}

In this section we describe sectional and normalized Ricci curvatures for the group of measure-preserving diffeomorphisms
of the projective plane. Since the sphere $\mb S^2$ is the orientation cover for $\mathbb{RP}^2$, we will start our investigation by finding a suitable basis for $\SVect_I(\mb S^2)$, similarly to Section~\ref{sec:klein.bottle.case}.

\subsection{Spherical harmonics and a basis for SVect$(\mathbb{RP}^2)$}
Divergence-free vector fields on the sphere can be described by their stream functions. As before, we consider a Laplace eigenbasis on the sphere.
A solution to the Laplace equation restricted to $\mathbb S^2$ is a linear combination of the functions
$Y^m_l(\theta,\phi)$ satisfying the eigenvalue problem
\begin{equation}\label{eq:eigenvalue problem}
\Delta Y^m_l(\theta,\phi)=-l(l+1)Y^m_l(\theta,\phi),\quad l\in \mb N_0, \quad  m \in \{ -l, \dots, l\},
\end{equation}
and called spherical harmonics.
They can be written in the form
\begin{equation}\label{eq:spherical harmonics}
Y^m_l(\theta,\phi)=\Big[\frac{(l-m)!}{(l+m)!}\frac{(2l+1)}{4\pi}\Big]^{1/2}e^{i\phi m}P^m_l(\cos\theta),
\end{equation}
where for $z=\cos\theta$ and $m \in \{ -l, \dots, l\}$
$$
P^m_l(z)=(1-z^2)^{m/2}\frac{1}{2^ll!}\frac{d^{l+m}}{dz^{l+m}}(z^2-1)^l
$$
are the associated Legendre polynomials\footnote{We do not include the Condon-Shortley phase $(-1)^m$ into the definition of the associated Legendre polynomials, nor in the definition of spherical harmonics. The same (acoustics communities) convention is used in~\cite{lukatskii_curvature_1979}.}. 

The normalization constants are chosen so that the complex spherical harmonics form an orthonormal system with respect to the Hermitian scalar product. Namely:
$$
\|Y^m_l(\theta,\phi)\|^2=\int_{\mathbb S^2}Y^m_l(\theta,\phi)\overline{Y^m_l(\theta,\phi)}\sin\theta\, d\theta d\phi=1,
$$
\begin{equation}\label{eq: inner-product}
\langle Y^m_l,Y^k_n\rangle=\int_{\mathbb S^2}Y^m_l(\theta,\phi)\cdot \overline{Y^k_n(\theta,\phi)}\sin\theta\, d\theta d\phi=\delta_{l,n}\delta_{m,k}.
\end{equation}
%
%
%
%
%
%
%
%

The group  $\SDiff(\mathbb S^2)$ consists of  measure-preserving diffeomorphisms of the sphere  $\mathbb S^2$ with
the area form $\mu=\sin\theta\, d\theta \wedge d\phi$. The corresponding Lie algebra $\SVect(\mathbb S^2)$ is formed by divergence-free vector fields. For computations below we use complex valued vector fields and then project them to the real subspaces. Extend the inner product on $\SVect(\mb S^2)$ to a Hermitian inner product on $\SVect(\mb S^2)\otimes{\mb{C}}$
\begin{equation}\label{eq: hermitian-product}
\langle u,v\rangle=\int_{\mathbb S^2}(u(p),\overline{v(p)})\,\mu(p),\quad u,{v}\in \SVect(\mathbb S^2)\otimes {\mb C},
\end{equation}
where $(u(p),\overline{v(p)})$ is the standard Hermitian product in $\mb C^3$, complexification of $\mb R^3\supset \mb S^2$.  
 The inner product on $\SVect(\mathbb S^2)$ defines the right-invariant $L^2$-metric on $\SDiff(\mathbb S^2)$.
\medskip

Stream functions $f_v={\rm sgrad}\, v $ for complexified vector fields $v\in \SVect(\mathbb S^2)\otimes {\mb C}$ are defined as in Section~\ref{sec:stream.func}. The spherical harmonics $Y^m_l$ form an orthonormal basis, denoted by $\mathcal Y$, for the Poisson algebra 
\[
\mf g^{\mb S^2}\otimes {\mb C} = \left\{ f: \mb S^2 \to \mathbb{C}\mid df = \mu(v_f, \cdot  ) = \sin\theta (-v_\phi d\theta + v_\theta d\phi)  \right\}.
\]
Vector fields ${\rm sgrad}\,Y^m_l$ are normalized using the inner product~\eqref{eq: hermitian-product} in $\SVect(\mb S^2)\otimes {\mb C}$, 
\begin{eqnarray*}
\int_{\mathbb S^2}{\rm sgrad}\,Y^m_l\cdot\overline{ {\rm sgrad}\, Y^m_l}\,d\mu&=&
l(l+1)\int_{\mathbb S^2}\big|Y^m_l\big|^2\,d\mu=l(l+1),
\end{eqnarray*}
where we used~\eqref{eq:eigenvalue problem}. Denote by
\begin{equation}\label{eq:eml}
e^m_l={\rm sgrad}\, \big(\frac{1}{\sqrt{l(l+1)}}Y^m_l\big)
\end{equation}
the elements of the orthonormal basis of $\SVect(\mathbb S^2)\otimes {\mb C}$.
\medskip

Recall that the orientation-reversing  involution
$I:{\mathbb S^2}\to {\mathbb S^2}$, defining  $\mathbb{RP}^2$, is the antipodal map given by
$$
I(x,y,z) = (-x, -y, -z)\quad \text{or} \quad
I(\theta, \phi) = ( \pi-\theta, \pi+\phi  )
$$
in different coordinate systems, see ~\eqref{eq:InvSphere}. 
The corresponding complexified Lie algebra is defined as the set of $I$-invariant vector fields:
$$  
\SVect(\mathbb{RP}^2) \otimes \mb C =\{v\in \SVect(\mb S^2) \otimes \mb C \ \vert\ I_*v = v\circ I \}  .
$$
The associated stream functions $f\in \mf g_{-1}^{\mb S^2} \otimes \mb C  $ are, respectively, $I$-anti-invariant, 
$f\circ I=-f$, see Lemma~\ref{lem:stream.func.composed.with.I}.

\begin{lemma}\label{lem:basis.vf.rp2}
The Lie algebra $\SVect(\mb{RP}^2)\otimes \mb C$ is isomorphic to the direct sum of the spaces of spherical harmonics with odd $l$ and equipped with the Poisson bracket. More precisely, the space  $\mathcal Y_l=\{\oplus_mY^m_l,\ m\in[-l,l]\},$ splits into the direct sum of the $I$-invariant and $I$-anti-invariant subspaces:
$$
\mf g^{\mb S^2}\otimes \mb C=\underbrace{\Big(\bigoplus_{l\ \text{odd}} \mathcal Y_l\Big)}_{\mf g^{\mb S^2}_{-1}\otimes \mb C}\oplus\underbrace{\Big(\bigoplus_{l\ \text{even}}  \mathcal Y_l\Big)}_{\mf g^{\mb S^2}_1 \otimes \mb C },
$$
and $\SVect(\mathbb{RP}^2)\otimes {\mb C} \simeq  \mf g_{-1}^{\mb S^2}\otimes \mb C.$
\end{lemma}
\begin{proof}
Since 
$$
\cos\theta\circ I=-\cos\theta,\quad \sin\theta\circ I=\sin\theta, \quad e^{-\phi}\circ I=-e^{-\phi},
$$
 the spherical harmonics with odd $l$ satisfy $f\circ I=-f$.
\end{proof}

\begin{corollary}
The group $\SDiff(\mb{RP}^2)$ with the Lie algebra $\SVect(\mb{RP}^2) \simeq  \mf g_{-1}^{\mb S^2} $ 
is a totally geodesic submanifold in the group $\SDiff(\mb S^2)$. In particular, its curvatures in two-dimensional planes 
containing odd spherical harmonics
can be computed  by using the curvatures in the ambient group $\SDiff(\mb S^2)$.
\end{corollary}

Note that unlike the Fourier basis of functions on the torus, 
the Poisson algebra of spherical harmonics is not a graded Lie algebra. This leads to difficulties with computing the sectional curvatures in the planes spanned by an arbitrary pair of spherical harmonics, cf.~\cite{arakelyan_geometry_1989} or~\cite{yoshida_riemannian_1997}.


\subsection{Sectional and Ricci curvatures of SDiff$(\mathbb{RP}^2)$ }
We first note that ``rotational vector field" $e^0_1=\sqrt{\frac{3}{8\pi}}\frac{\p}{\p \phi}$ descends from $\mb S^2$ to $\mb{RP}^2$, while
for the planes passing through it the sectional curvatures 
of $\SDiff(\mb S^2)$  (and hence of $\SDiff(\mb{RP}^2)$) were found in ~\cite{lukatskii_curvature_1979}:
$$
C(e^0_1,e^m_l)=\frac{3m^2}{8\pi \, l^2(l+1)^2}\,.
$$
Below we describe the sectional curvature of $\SDiff(\mb{RP}^2)$ for planes containing the ``trade wind'' 
 vector field $e^0_3$, as well as compute  asymptotic and normalized Ricci curvatures in those directions. 

\begin{figure}
\includegraphics[width=0.7\linewidth]{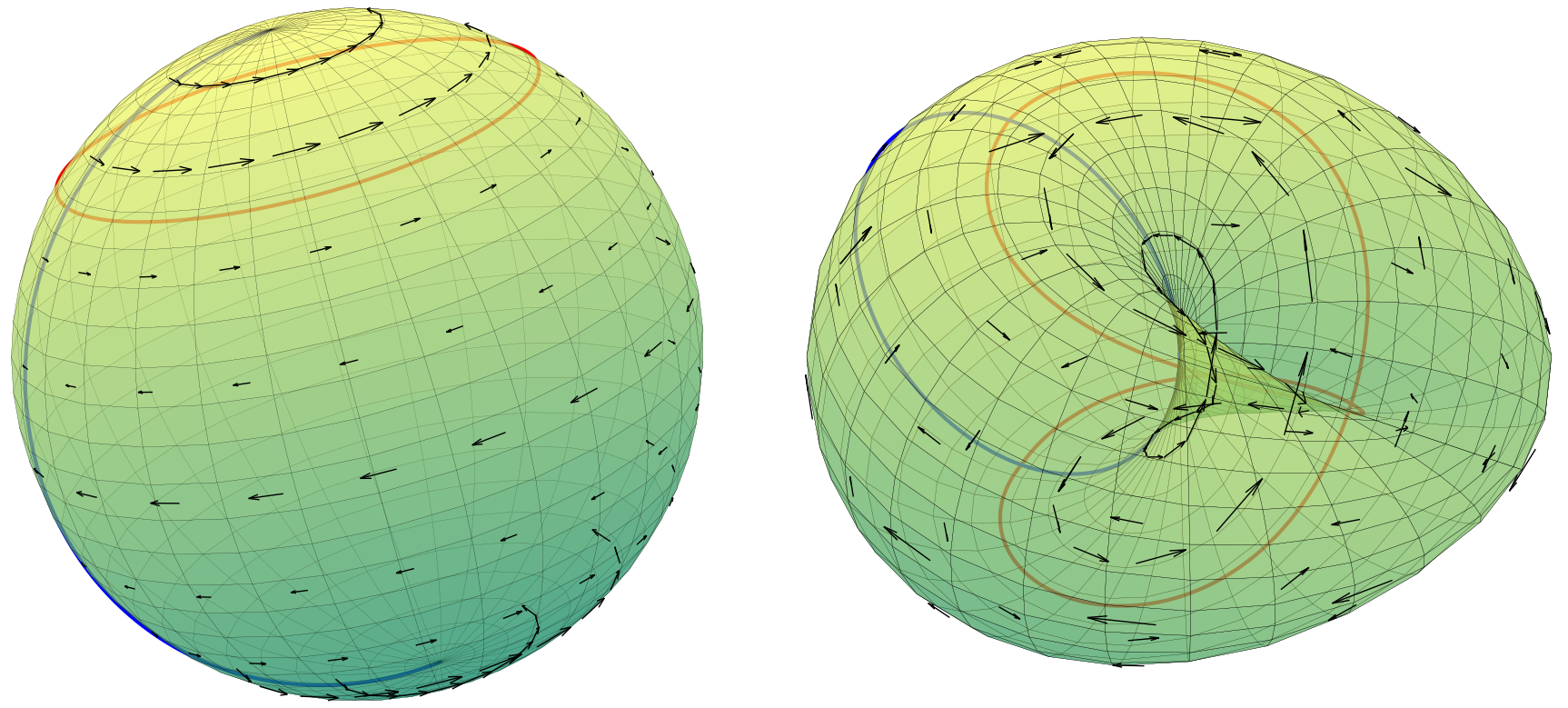}
\centering \includegraphics[width=0.45\linewidth]{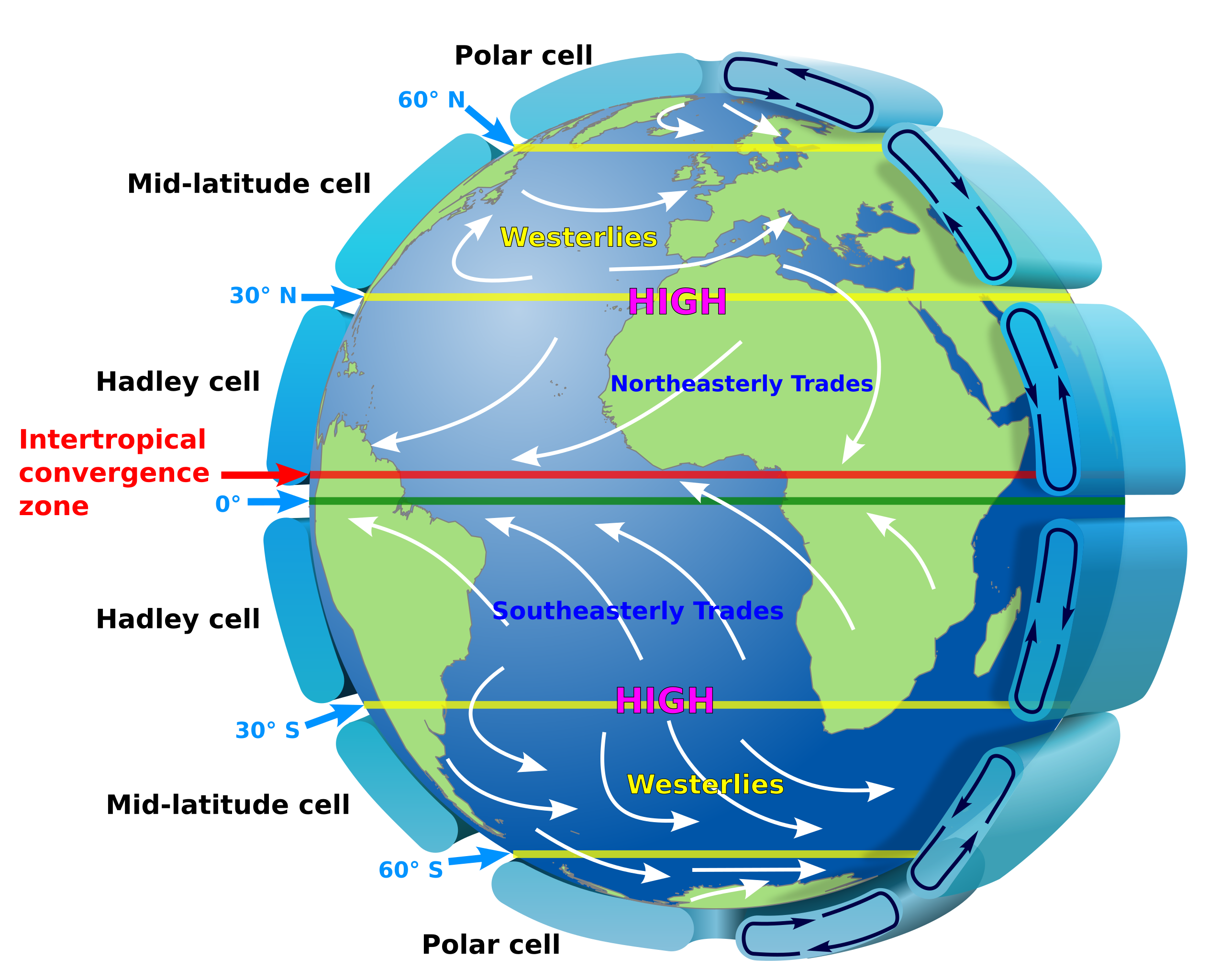}
\caption{ \tiny The vector field $e^0_3 = \frac{1}{8}\sqrt{\frac{21}{\pi}}(5\cos^2(\theta) -1)\frac{\p }{\p \phi} $ on the sphere in the subalgebra $\SVect_I(\mb S^2)$ gives a vector field on $\mathbb{RP}^2$. The visualization of $\mb{RP}^2$ is given by the so-called cross-cap immersion. The lower figure is an illustration of the ``trade wind'' on Earth. Figure attribution: Kaidor, CC BY-SA 3.0 via Wikimedia Commons.}
\label{fig:sphere_tradewind_new}
\end{figure}

Consider the basis vector field $e^0_3$ in $\SVect_I(\mb S^2) \otimes \mb C$, obtained by taking the skew-gradient of the spherical harmonic $Y^0_3 = \frac{1}{4}\sqrt{\frac{7}{\pi}} \big( 5\cos^3\theta -3\cos\theta \big)$:
\begin{equation}\label{eq:sphere.tradewind}
e^0_3= \frac{1}{\sqrt{12}} \sg(Y^0_3)= \frac{1}{8}\sqrt{\frac{21}{\pi}}\big( 5\cos^2\theta - 1 \big)\frac{\p}{\p \phi}.
\end{equation}
see Figure~\ref{fig:sphere_tradewind_new}. 

\begin{theorem}\label{th:sec.curv.cpe}
The sectional curvature of $\SDiff(\mb{RP}^2)$ in the planes spanned by the vector field $e^0_3$ and the basis element $e^m_l$ is given by
\begin{equation}\label{eq:CT2}
C(e^0_3,e^m_l)=-\frac{525m^2}{32\pi}\Big( 3\gamma^m_l - 2\Big(\frac{12}{l(l+1)}+1\Big)\varrho^m_l - \Big(\frac{12}{l(l+1)}-1\Big)^2\vartheta^m_l \Big),
\end{equation}
where
\begin{small}
\begin{eqnarray*}
\vartheta^m_l&=& \frac{1}{8}\frac{\big((l+2)^2-m^2\big)\big((l+1)^2-m^2\big)(l+1)l}{(2l+3)^2(2l+5)(2l+1)(l+2)(l+3)}
\\
+\frac{1}{8}
\frac{\big((l-1)^2-m^2\big)\big(l^2-m^2\big)(l+1)l}{(2l-3)(2l-1)^2(2l+1)(l-2)(l-1)}
&+&\frac{1}{8}
\Big(\frac{l^2-m^2}{(2l+1)(2l-1)}+\frac{(l+1)^2-m^2}{(2l+3)(2l+1)} -\frac{1}{5} \Big)^2\,,
\end{eqnarray*}
\begin{eqnarray*}
\varrho^m_l&=&\frac{1}{8}
\frac{\big((l+2)^2-m^2\big)\big((l+1)^2-m^2\big)}{(2l+1)(2l+3)^2(2l+5)}
\\
+\frac{1}{8}
\frac{\big((l-1)^2-m^2\big)\big(l^2-m^2\big)}{(2l-3)(2l-1)^2(2l+1)}
&+&\frac{1}{8}
\Big(\frac{l^2-m^2}{(2l+1)(2l-1)}+\frac{(l+1)^2-m^2}{(2l+3)(2l+1)} -\frac{1}{5} \Big)^2\,,
\end{eqnarray*}
\begin{eqnarray*}
\gamma^m_l&=&\frac{1}{8}
\frac{\big((l+2)^2-m^2\big)\big((l+1)^2-m^2\big)(l+2)(l+3)}{(2l+1)(2l+3)^2(2l+5)(l+1)l}
\\
+\frac{1}{8}
\frac{\big((l-1)^2-m^2\big)\big(l^2-m^2\big)(l-2)(l-1)}{(2l-3)(2l-1)^2(2l+1)(l+1)l}
&+&\frac{1}{8}
\Big(\frac{l^2-m^2}{(2l+1)(2l-1)}+\frac{(l+1)^2-m^2}{(2l+3)(2l+1)} -\frac{1}{5} \Big)^2\,.
\end{eqnarray*}
%
\end{small}
\end{theorem}

\begin{tiny}
\begin{table}
\centering
\begin{tabular}{l|lllllllr}
 & $l=3$ & $l=4$ & $l=5$ & $l=6$ & $l=7$ & $l=8$ & $l=9$ & $l=10$ \\ \hline
$m= \pm 1$ & 0.010 & 0.001 & 0.000 & 0.000 & -0.000 & -0.000 & -0.000 & -0.000 \\
$m= \pm 2$ & -0.172 & -0.072 & -0.034 & -0.018 & -0.010 & -0.006 & -0.004 & -0.003 \\
$m= \pm 3$ & -0.172 & -0.283 & -0.164 & -0.094 & -0.056 & -0.035 & -0.023 & -0.016 \\
$m= \pm 4$ &  & -0.190 & -0.343 & -0.241 & -0.157 & -0.103 & -0.070 & -0.048 \\
$m= \pm 5$ &  &  & -0.194 & -0.373 & -0.299 & -0.215 & -0.152 & -0.109 \\
$m= \pm 6$ &  &  &  & -0.190 & -0.385 & -0.339 & -0.263 & -0.198 \\
$m= \pm 7$ &  &  &  &  & -0.184 & -0.387 & -0.367 & -0.302 \\
$m= \pm 8$ &  &  &  &  &  & -0.177 & -0.383 & -0.385 \\
$m= \pm 9$ &  &  &  &  &  &  & -0.169 & -0.376 \\
$m= \pm 10$ &  &  &  &  &  &  &  & -0.161 \\
\end{tabular}
\caption{ \tiny $C(e^0_3, e^m_l)$ from Theorem~\ref{th:sec.curv.cpe} for $3\leq l \leq 10$. }
\label{tab:Cpelm.table}
\end{table}
\end{tiny}

Numerical values for small $l$ are shown in Table~\ref{tab:Cpelm.table}. One can see that, while most of the curvatures are negative, the first row contains positive values. We address this observation in the next theorem.

The proof of Theorem~\ref{th:sec.curv.cpe} is a rather tedious computation. It uses 
the following symmetry property, which is of independent interest. Recall that the space $\mathcal Y_l$ consists of 
spherical harmonics with eigenvalues $-l(l+1)$. 

\begin{lemma}\label{lem:B.operator.cond} Let $u \in \mathcal Y_l$ and $v\in \mathcal Y_k$.
The operator $B$ defined by $\langle B(u,v) ,w \rangle  = \langle u , [v,w] \rangle$ in $\SVect(\mb S^2)$ satisfies
\[		{k(k+1)}\,B(u,v)=-{l(l+1)}\,B(v,u).\]
\end{lemma}

The proof of lemma is based on the fact that for a two-dimensional manifold 
$B(u,v)=-\Delta f_u\nabla f_v+\nabla h,$
where $\sg(f_u)=u$, $\sg(f_v)=v$, and $h$ is a function uniquely defined by the condition that $\diverg B(u,v)=0$, see~\cite{arnold_mathematical_1989}.  Now the result follows from the fact that for $u\in V_l$ one has 
$\Delta f_u=-l(l+1)f_u$. 
\medskip

In order to compute  the normalized Ricci curvature of $\SDiff(\mb S^2)$ in the direction of the ``trade wind'' vector field $e^0_3$ in~\eqref{eq:sphere.tradewind}, consider the basis vectors $e_l^m = \frac{1}{\sqrt{l(l+1)}}\sg\left(Y^m_l\right)$. Since $-l\leq m \leq l$, the quotient $\frac{m}{l}$ gives the slope of the straight line through the origin and the point $(l,m)\in  \mb N_0^2$, see Figures~\ref{fig:cone.lattice} and \ref{fig:blanket.rp2}.

\begin{figure}
\begin{minipage}[t]{0.5\textwidth}
\centering
\includegraphics[width=\textwidth]{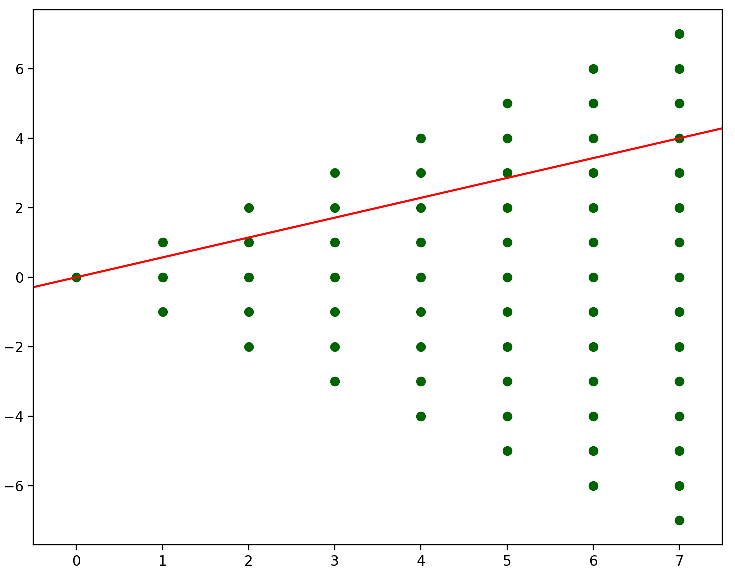}
\caption{ \tiny Lattice of the points $(l,m)$ where ${|m|\leq l}$. The line intersects $(0,0)$ and $(7,4)$ and by using Theorem~\ref{th:asymp.sec.curv.cpe} with $q=\frac{4}{7}$ we get that
${\lim\limits_{l\to \infty} C(e^0_3,e_l^m) }  \simeq -0.3749 $
for elements $e_l^m$ along the line.}
\label{fig:cone.lattice}
\end{minipage}%
\begin{minipage}[t]{0.5\textwidth}
\centering
\includegraphics[width=\textwidth]{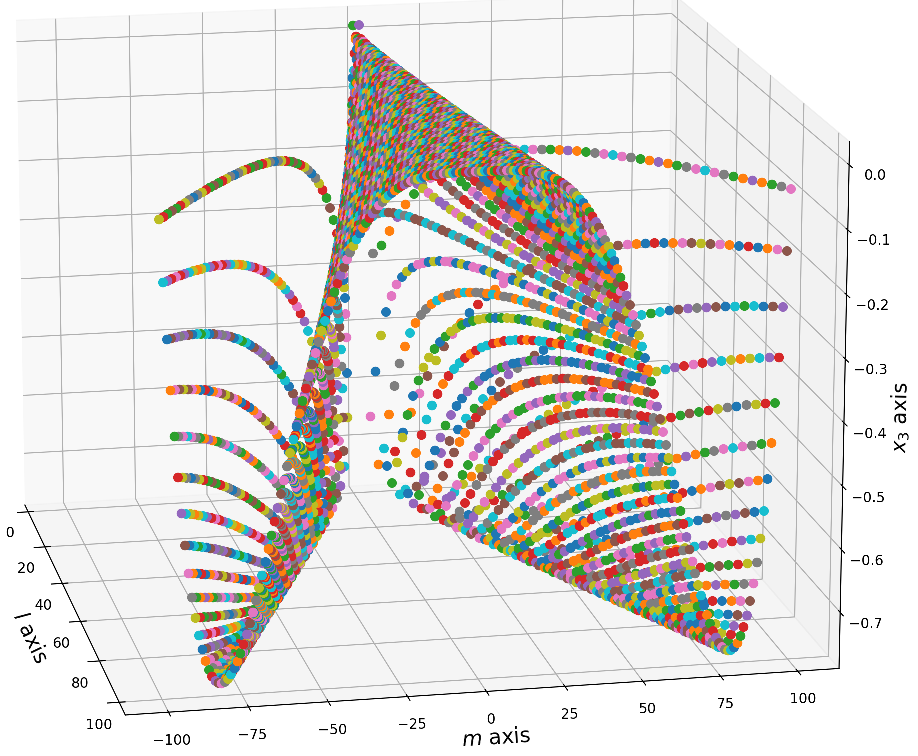}
\caption{ \tiny The value of $C(e^0_3 , e_l^m)$ along the $x_3$-axis for $e^0_3$ fixed and $e_l^m$ with $2\leq l\leq 100$ and $|m|\leq l$. For fixed large $l$, $C(e^0_3,e_l^m)$ is approximately given by $-\frac{525}{32\pi} (m/l)^4(1-(m/l)^2)$, see Theorem~\ref{th:asymp.sec.curv.cpe}. The colours are to distinguish individual points.}
\label{fig:blanket.rp2}
\end{minipage}
\end{figure}

\begin{theorem}\label{th:asymp.sec.curv.cpe}
Consider a sequence of basis elements $\{e^{m_l}_l\}_{l\in \mb N}$, where $|m_l| \leq l$ and ${\lim\limits_{l\to \infty} \frac{m_l}{l} = q}$. Then
\[\lim\limits_{l\to \infty} C(e^0_3, e^{m_l}_l) = -\frac{525}{32\pi}q^4(1-q^2) \,. \]
\end{theorem}
\begin{corollary}
\begin{enumerate}
\item The sectional curvature $C(e^0_3, e^m_l)$ is positive only for $m= \pm 1$, and $\lim\limits_{l\to \infty} C(e^0_3, e^{\pm 1}_l) = 0.$ 
\item The infimum of the sectional curvature $C(e^0_3, e^{m}_l)$ is equal to $ -\frac{175}{72\pi}$. The sequences $\{ C(e^0_3, e^{m_l}_l )\}$ converges to the infimum, when $\lim\limits_{l \to \infty} \frac{m_l}{l} = \pm \sqrt{\frac{2}{3}}$ . 
\end{enumerate}
\end{corollary}
\begin{proof}[Proof of Theorem \ref{th:asymp.sec.curv.cpe}]
Rewrite the expression for $C(e^0_3,e_l^{m_l})$ from Theorem~\ref{th:sec.curv.cpe} in the form
\begin{equation*}
C(e^0_3, e_l^{m_l})  = -\frac{525}{32 \pi} \left( - h_l^2{m_l}^2\vartheta^{m_l}_l   + 2h_l{m_l}^2(\varrho^{m_l}_l - \vartheta^{m_l}_l)   + {m_l}^2 ( 3\gamma^{m_l}_l -2\varrho^{m_l}_l -\vartheta^{m_l}_l  )\right),
\end{equation*}
where $h_l = \frac{12}{l(l+1)}$. Since as $l\to \infty$ the terms with $h_l$ and  $h_l^2$ vanish, one gets
\begin{equation}\label{eq:asymp.cpe.formula}
C(e^0_3,  e_l^{m_l})  =-\frac{525{m_l}^2}{32 \pi} \left(  3\gamma^{m_l}_l  -2\varrho^{m_l}_l  -\vartheta^{m_l}_l \right) + O(1/l).
\end{equation}
Expanding in powers of $l$ and ${m_l}$ and using that  $ \frac{m_l}{l} = q + O(1/l)$, after tedious computations, one obtains
$$
C(e^0_3, e_l^{m_l})  =- \frac{525}{32 \pi} q^4(1-q^2) + O(1/l).
$$
\end{proof}

\begin{corollary}\label{cor:ricci.rp2}
The normalized Ricci curvature of $\SDiff(\mb S^2)$ and $\SDiff(\mb{RP}^2)$ in the direction $e^0_3$ is
\[\Ric(e^0_3) = -\frac{15}{8\pi } = -\frac{15}{2S_{\mb S_2}} = -\frac{15}{ 4 S_{\mb{RP}^2 } }. \]
\end{corollary}
\begin{proof}
First note that for $l\leq R$, there are $\sum_{l=0}^{R}(2l+1) = (R+1)^2$  elements in the basis $\left\{ e^m_l \right\}$. 
According to Definition~\ref{def:norm.ric.curv} we need to find 
\[\Ric(e^0_3) = \lim\limits_{R\to \infty} S_R =\lim\limits_{R \to \infty} \frac{1}{(R+1)^2 } \sum_{l=0}^{R} \sum_{m=-l}^{l} C(e^0_3, e^m_l), \quad e^m_l = \frac{1}{\sqrt{l(l+1)}}\sg(Y^m_l) .\]
As in the proof of Theorem~\ref{th:ric.klein}, we remove a finite set of elements from the sum with $l < R_1$ without changing the value of the sequence $S_R$, such that for every $l\geq R\geq R_1$ one can apply Theorem~\ref{th:asymp.sec.curv.cpe}. Thus for $q = \lim\limits_{l\to \infty} \frac{m}{l}$ one has
\[ \Ric(e^0_3) = \lim\limits_{R\to \infty} -\frac{1}{(R+1)^2 } \sum_{l=R_1}^{R} \frac{2l+1}{2l+1} \sum_{m=-l}^{l}\frac{525}{32\pi}q^4(1-q^2)\,.  \]
 For each $l$, the sum 
\[ \frac{1}{2l+1} \sum_{m=-l}^{l} \frac{525}{32\pi}\left(\frac{m}{l}\right)^4 \left(1-\left(\frac{m}{l}\right)^2 \right)  \] 
is a Riemann sum over an evenly spaced partition. The integral of the corresponding  function is $15/8\pi$, hence the Riemann sums will converge to $15/8\pi$ as $l\to \infty$. Therefore
\[ \Ric(e^0_3) = \lim\limits_{R\to \infty} -\frac{15}{8\pi}\frac{1}{(R+1)^2 } \sum_{l=R_1}^{R} (2l+1) = -\frac{15}{8\pi}.  \]

When calculating the normalized Ricci curvature in $\SDiff(\mb{RP}^2)$ we only consider the basis elements $e^m_l$ where $l$ is odd. This removes half of the basis elements, but does not change the limit value of the average.
\end{proof}

\subsection{Ricci curvature of SDiff$(\mathbb S^2)$}
Recall that Theorem~\ref{th:sec.curv.cpe}  applies to both $\SDiff(\mb{RP}^2)$ and $\SDiff(\mb S^2)$, giving an explicit formula for  sectional curvatures in planes containing the vector field  ${e^0_3 =\frac{1}{8}\sqrt{\frac{21}{\pi}}\big( 5\cos^2\theta - 1 \big)\frac{\p}{\p \phi} }$ on the sphere. In~\cite{lukatskii_curvature_1979}, Lukatskii obtained a similar formula for  sectional curvatures in planes containing the vector field $e^0_2 =\sqrt{\frac{15}{8\pi}}  \cos\theta \frac{\p}{\p \phi}$. Namely, those curvatures are given by
\begin{equation}\label{eq:lukatskii.formula}
\resizebox{0.92\hsize}{!}{$%
C(e^0_2 , e^m_l) = \frac{15m^2}{32\pi} \left( (1-c_l^2)\big(a^m_l b_l + \frac{a^m_{l+1}}{b_{l+1}}\big) +2(1+c_l)(a^m_l + a^m_{l+1})-3\big(\frac{a^m_l}{b_l} + a^m_{l+1 } b_{l+1}\big)  \right)\,,%
$}
\end{equation}
where \[c_l = \frac{6}{l(l+1)}, \quad b_l = \frac{l+1}{l-1}, \quad a^m_l = \frac{l^2-m^2}{4l^2-1}. \]
The vector field $e^0_2$ on the sphere $\mb S^2$ does not satisfy $I_*v = v\circ I$ since $l$ is even, and  therefore it does not descend to a vector field on $\mb{RP}^2$. 
Yet, it is interesting to compare  the values of the normalized Ricci curvatures $\Ric(e^0_2)$ with $\Ric(e^0_3)$
for the sphere case. The following theorem gives the asymptotic values of sectional curvatures $C(e^0_2, e^{m_l}_l)$
 as $l\to \infty$ for sequences  with $\lim\limits_{l\to \infty}\frac{m_l}{l} = q$. 

\begin{theorem}\label{th:asymp.sec.curv.cge}
Consider a sequence of basis elements $\{e^{m_l}_l\}_{l\in \mb N}$, where $|m_l|\leq l$ and  ${\lim\limits_{l\to \infty} \frac{m_l}{l} = q}$. Then
\[\lim\limits_{l\to \infty} C(e^0_2, e^{m_l}_l) = -\frac{15}{8\pi}q^4\,.  \]
\end{theorem}
\begin{proof} 
In \eqref{eq:lukatskii.formula} denote 
\[  A_0 = \frac{a^m_l}{b_l} + a^m_{l+1 } b_{l+1}, \quad  A_1 = a^m_l + a^m_{l+1}, \quad A_2 = a^{m}_lb_l + \frac{a^{m}_{l+1} }{b_{l+1}}.\]
Then Equation \eqref{eq:lukatskii.formula} can  be rewritten as a second order polynomial in $c_l = O(1/l^2)$:
\[ C(e^0_2, e^{m_l}_l) = \frac{15m_l^2}{32\pi}\left((A_2 +2A_1 -3A_0) + 2c_l (A_1 - A_2)+c_l^2 A_2  \right) . \]
As $l\to \infty$ the terms with $c_l$ and  $c_l^2$ both vanish. Expanding the rest, using $\frac{m_l}{l} = q + O(1/l)$
and simplifying, one obtains:
\[C(e^0_2, e^{m_l}_l)  =  \frac{15m_l^2}{32\pi}\Big( -\frac{4q^2}{l^2} \Big) + O(1/l)  = -\frac{15}{8\pi}q^4+ O(1/l) .\]
\end{proof}
\begin{corollary}\label{cor:ricci.s2.g.direction}
The normalized Ricci curvature of $\SDiff(\mb S^2)$ in the direction $e^0_2$ is given by
\[\Ric(e^0_2) = -\frac{3}{4\pi} = -\frac{3}{S_{\mb S^2}}.\]
\end{corollary}



\section{Unreliability of weather forecasts on the Klein bottle and  projective plane}\label{sect:weather}
Following Arnold~\cites{arnold_mathematical_1989, arnold_topological_2021} we make the following assumptions:
\begin{itemize}
\item The earth has the shape of the Klein bottle or  projective plane, respectively.
\item The atmosphere is a two-dimensional homogeneous incompressible inviscid fluid.
\item The motion of the atmosphere is approximately a ``trade wind,'' which we will model by an appropriate vector field $v$.
\end{itemize}
For now, let $M$ be a compact Riemannian surface. If a surface is negatively curved, this leads to the exponential divergence of the geodesics on it.
The {\it characteristic path length} is the average path length on which a small error in the initial condition of a geodesic on 
$\SDiff({M})$ grows by a factor of $e$, see~\cite{arnold_mathematical_1989}. For a geodesic 
with an initial vector $v$ its characteristic path length 
is given by $s_0 = {1}/{\sqrt{-C_0(v)}}$,
where $C_0(v)$ is the ``average sectional curvature" in the planes containing $v$. Below we 
assume that this average curvature on the group $\SDiff({M})$ is the  Ricci curvature $\Ric(v)$ in the $v$-direction and then
$$
s_0 = {1}/\sqrt{-\Ric(v)}\,.
$$

Recall that Arnold's reasoning in~\cite{arnold_mathematical_1989} for the instability of the earth atmosphere is based on the following consideration. Let a vector field $v$ on $M$ have  $L^2$-norm $\norm{v} $ and  the ``average sectional curvature" $C_0(v)$.
Then the time it takes for our fluid or atmospheric flow (corresponding to the geodesic on $\SDiff({M})$ with initial vector $v$)  to travel the  characteristic path length $s_0$ is the characteristic time $t_0:=s_0/\norm{v} $, while for this time the errors grow by a factor of $e$.

One also assumes that the vector field $v$ has many periodic trajectories on $M$, and its fastest particles have a typical period
$T:=s_\text{orb}/|v|_\text{rapid}$, where $s_\text{orb}$ stands for typical lengths of periodic orbits of those particles with speed $|v|_\text{rapid} $. 
Then during the characteristic time $t_0$ the fastest particles make $t_0/T$ part of the full period.
This implies that the error in the initial condition will grow by a factor of $e^\alpha$,
where 
\begin{equation}\label{eq:alpha}
\alpha:=\frac{T}{t_0}= \frac{s_\text{orb}/|v|_\text{rapid}}{s_0/\norm{v} }=
\frac{s_{\text{orb}}\norm{v} \sqrt{-\Ric(v)} }{|{v}|_{\text{rapid}}  } \,,
\end{equation}
after each full period of the fastest particle.

To apply this consideration to the earth-like atmosphere one needs to rescale the units. In~\cite{arnold_mathematical_1989}
one assumes that the ``equator's circumference" of the compact Riemannian surface ${M}$ is equal to the earth's equator of $ 40~000$km, which will be the orbit length $s_{\text{orb}}$ of the fastest particles, 
and sets the vector field $v$ to be a ``trade wind'' current with maximal velocity of $100$km/h. 
This implies that the orbit time (i.e. the period) for the fastest particles in the earth atmosphere is $400$ hours, and hence
the number of periods of those particles per month is 
$$
n= \frac{30\cdot 24}{T_{\text{in hours}} } = \frac{30\cdot 24}{ \frac{40~000\text{km}}{100\text{km/h}} }
$$
Therefore, if at the initial moment, atmospheric measurements are known with an error $\epsilon$, the magnitude of the error of prediction after $N$ months would be $\epsilon \cdot 10^{kN}$, where $k$ is
\[k = \alpha \cdot n \cdot \log_{10}e\, . \]
Thus the value of $k$ tells us how many more digits of accuracy we need to know today in order to predict the weather on $M$
for 1 month  for a typical trade wind $v$. 

\begin{remark}
{\rm 
The values  of $\alpha$ and $k$ depend on properties of the chosen ``trade wind''. In particular,  $\alpha$ is proportional to the ratio of the average speed of all particles over  the fastest ones. Hence the magnitude of error will also depend on this ratio. 

Recall that for  a vector field $v$ on a non-orientable $ M$ we use the isometric normalization, i.e. $\norm{v}^2 =\norm{\tilde v}^2 $ for  the lifted field $\tilde v$ on the orientation double cover $\tilde M$. (Note that the value of $\alpha$ is not affected by the normalization, as any rescaling changes both the norm of $v$ and the sectional curvatures, but not the product  $\norm{v}\sqrt{-\Ric(v)}$, see Remark \ref{rem:norm}.)
\smallskip

However, for a non-orientable manifold there arises another interesting phenomenon.
Recall that $ |{v}|_{\text{rapid}}$ is  defined as the maximal average
speed $ |{v}|_{\text{rapid}} = \sup \{ \frac{1}{b-a} \int_a^b |\dot{\gamma}(t))|dt\}$ on $v$-trajectories $\gamma$.
We will be looking at the periodic orbits of the field $v$ on $M$. Such closed curves on $M$ can be of two types: either $i)$ the orientation of $M$ does not change after travelling along the orbit, or  $ii)$ it does. 
In the  case $i)$, the fastest particles of the trade wind  fly mostly along orbits that do not change the orientation of $M$, and one may assume that the lengths of their orbits on $M$ and on its cover $\tilde M$ coincide, $\tilde s_\text{orb}= s_\text{orb}$.

On the other hand, in the case $ii)$,  
a neighbourhood of  an orientation-changing orbit looks like a M\"obius band, while the orbit itself is half as long as the neighbouring closed ones. Respectively the lift of such an orbit to $\tilde M$ is not closed, being a half of the closed 
orbit for $\tilde v$. 
Thus if the orbits of the fastest particles lie in the vicinity of an orientation-changing trajectory, one may assume that 
$\tilde s_\text{orb}= 2 \, s_\text{orb}$, i.e. the length of the equator of $M$ (for the trade wind field $v$) is equal to half the length of the equator of the cover $\tilde M$. 
}
\end{remark}


\begin{remark}
{\rm 
Intuitively the exponent $k$ should not depend on the length of the orbit of the fastest particle $s_{\text{orb}}$: indeed,  $\alpha$ is proportional to $s_{\text{orb}}$, as doubling the orbit length doubles the number of characteristic times per period, while the number of periods of the fastest particles per month is inversely proportional to the orbit length $s_{\text{orb}}$.

What it does depend upon is the rescaling factor $c_M$ relating the metric on the manifold $M$ with distance on the earth, 
$\hat g = {c^2_M} g$. Combining this with  the above one obtains, after certain simplifications,
$$
k=\frac{ \norm{v} \sqrt{-\Ric(v)} }{|{v}|_{\text{rapid}} \,   c_M}\cdot 30\cdot 24\cdot 100 \log_{10}e \, .
$$

Note that there are different ways to introduce rescaling factor $c_M$. For instance, Arnold suggested that the equator $E_M$ of the manifold $M$,  should have length $40~000$km, and hence $c_M = {40~000\text{km}}/{E_M}$. The value of the equator length for a non-orientable manifold is an ambiguous notion, as it can be defined by longest loops preserving or changing its orientation. 
Alternatively, one can relate the areas of $M$ and the earth, which might be more relevant for the non-orientable case, as we discuss below.
}
\end{remark}


\subsection{Weather forecasts on the Klein bottle}
If the earth were of the shape of the Klein bottle a natural model for the ``trade wind'' would be the vector field (corresponding to a stream function)
\begin{equation}\label{eq:tradewind}
v = \sin(x_2)\frac{\p}{\p x_1} \quad \longleftrightarrow \quad \xi_{(0,1)}= -\cos(x_2) = -\frac{1}{2}(e_{(0,1)} + e_{(0,-1)} ),
\end{equation}
see Figure~\ref{fig:vf}. By Theorem~\ref{th:ric.klein} the normalized Ricci curvatures of $\SDiff(\mb K)$ in the planes containing the ``trade wind'' $\xi_{(0, 1)}$ is $\Ric_{\mb K}(\xi_{(0, 1)}) =  -\frac{3}{16 S_{\mb K}}  $. The normalized Ricci curvature of $\SoDiff(\mb T)$ in the same direction is $\Ric_{\mb T}(\xi_{(0, 1)}) = -\frac{3}{8 S_{\mb T} }$, since $S_{\mb T} = 2S_{\mb K}$, also see~\cite{lukatskii_curvature_1984}.

The $L^2$-norm of $v$ on $\SDiff(\mb K)$ is $\norm v^2=S_{\mb T}/2=S_{\mb K}$.
The fastest particles in the ``trade wind''~\eqref{eq:tradewind} on the Klein bottle have the speed  $|{v}|_{\text{rapid}} = 1$. 
An orbit on ``the equator'' of the Klein bottle (or on the torus) can be parametrized by 
\[\gamma(t) = (x_1(t), \, x_2(t))  = (t, \, \pi/2) , \quad t\in [0,2\pi] \,.\]

Following \cite{arnold_mathematical_1989}, we assume that the ``equator's circumference'' of the torus and Klein bottle is equal to the earth's equator of $40~000$km. Since $\gamma$ has length $2\pi$ on $\mb K$ and $\mb T$, we get the rescaling factor $c_{\mb K} = c_{\mb T} = {40~000\text{km}}/{2\pi}. $
Therefore for initial atmospheric measurements with an error $\epsilon$, the magnitude of the error of prediction after $N$ months would be $\epsilon \cdot 10^{kN}$, where $k$ is
\begin{equation*}
k_{\mb K} = k_{\mb T} = \frac{\sqrt{\frac{S_{\mb T}}{2}} \sqrt{\frac{3}{8 S_{\mb T}} } }{ 1\cdot \frac{40~000}{2\pi}  } 30\cdot 24\cdot 100\log_{10}e \simeq 2.1 \, .
\end{equation*}
We conclude that to predict the weather on the Klein bottle (or on the torus) with the ``equator's circumference'' equal to that of the earth's, for such a typical trade wind, for 2 months, one needs to know it today with $4$ more digits of accuracy. 

\begin{remark}
{\rm 
In~\cite{arnold_mathematical_1989}, Arnold obtained $k_{\mb T} \simeq 2.5$ (and hence  $5$ more digits of accuracy
for a two-month weather prediction). This was based on estimating the ``average sectional curvature'' $C_0(v)$ in planes containing the ``trade wind'' $v=\xi_{(0,1)}$ to be $C_0 = -\frac{1}{2 S_{\mb T}}$. Above, instead, we used the normalized Ricci curvature
to interpret the ``average sectional curvature''  value $C_0 = \Ric(\xi_{(0,1)}) = -\frac{3}{8 S_{\mb T}}$. 
}
\end{remark}


\subsection{Weather forecasts on the real projective plane}
If the earth were of the shape of the real projective plane it would be natural to  model  the ``trade wind'' by the vector field 
\begin{equation}\label{eq:tradewind.rp2}
v= e^0_3 =\frac{1}{8}\sqrt{\frac{21}{\pi}}\big( 5\cos^2\theta - 1 \big)\frac{\p}{\p \phi} \,,
\end{equation}
see Figure~\ref{fig:sphere_tradewind_new}. 
This field is normalized, so that $\norm{v}^2=1$. The normalized Ricci curvature of $\SDiff(\mb{RP}^2)$ in the planes containing this vector field is $\Ric_{\mb{RP}^2} (e^0_3) = -\frac{15}{4S_{\mb{RP}^2}} $, according to Corollary~\ref{cor:ricci.rp2}.
The normalized Ricci curvatures of $\SDiff(\mb S^2)$ in the same direction is 
$  \Ric_{\mb{S}^2} (e^0_3) = -\frac{15}{2 S_{\mb{S}^2}} ,$ since $2S_{\mb{RP}^2}=S_{\mb{S}^2}$.
The orbits of $v = e^0_3$ are given by ``circles of latitude" $z={\rm const}$ on the sphere, or $\theta={\rm const}$ in spherical coordinates, and on each orbit the speed is constant. 

The ``fastest particles'' are the ones with highest average speed:  they  maximize\\ ${\frac{1}{b-a}\int_a^b |5\cos^2(\theta) -1|\sin (\theta) \, dt }$. 
One can show that the fastest particles in the ``trade wind''~\eqref{eq:tradewind.rp2} on the sphere and on the real projective plane  (along orbits not changing orientation) have the speed  $|{v}|_{\text{rapid}}  = \frac{2}{3} \sqrt{\frac{14}{5 S_{\mb{RP}^2 } }}$. The orbits of such fastest particles have length $\frac{4\pi }{\sqrt{15}} $ at the latitude $z= \frac{2}{\sqrt{15}} $, see Figure~\ref{fig:sphere_tradewind_new}. 

However, note that the equator for ${\mb{RP}^2 }=\mb{S}^2/I$ is equal to $\pi$: it is  twice as short 
as that of $\mb{S}^2$, as there are   orbits of the field $v$ along which the orientation changes. They are not the fastest ones, but they affect the equator length, and hence the scaling factor $c_M$. 
\smallskip

Computations for $\mb{S}^2 $ and ${\mb{RP}^2 }$ (similar to the above for $\mb{T}$ and ${\mb{K}}$) give that 
$k_{\mb S^2} \simeq 8.5$ and $k_{\mb{RP}^2} \simeq 4.3$. Note that these values differ by a factor of 2 due to the above-mentioned equator shortening. As we discussed the corresponding values of $k$ describe the loss of  accuracy in weather predictions in terms of the number of digits per month. 

\medskip

The above  computations are summarized in the following table:

\begin{center}
\begin{tabular}{l | l l l l }
${ M } $ 	& $\mb K$ 		&	$\mb T$  		&  $\mb{RP}^2$  &  $\mb{S}^2$\\
$\norm{v}$ 	& $\sqrt{\frac{S_{\mb T}}{2}} $		& $\sqrt{\frac{S_{\mb T}}{2}} $				& $1$				& $1$\\[7pt]
$\Ric(v)$	&		$-\frac{3}{8S_{\mb T}}$		&	$-	\frac{3}{8S_{\mb T}}$					&				$-\frac{15}{2S_{\mb{S}^2 } }$			&$-\frac{15}{2S_{\mb{S}^2 } }$			\\[7pt]
$|v|_{\text{rapid}}$	& $1$						&				$1$			&	$\sqrt{\frac{28}{5S_{\mb S^2}}}$						&$\sqrt{\frac{28}{5S_{\mb S^2}}}$		\\[7pt]
${E_M}$ 	& $2\pi $		& $2\pi $				& $\pi $				& $2\pi $\\
$c_M$  &	 $\frac{40~000\text{km}}{2\pi}$& $\frac{40~000\text{km}}{2\pi}$ & $\frac{40~000\text{km}}{\pi}$	&$\frac{40~000\text{km}}{2\pi}$			\\[7pt]
$k$ &  $2.1$ & $2.1$  & $4.3$  &  $8.5$
\end{tabular}
\end{center}
\medskip

Thus assuming the above trade wind models, we conclude that to predict the weather on the projective plane for 2 months one needs to know it today with $8$ more digits  of accuracy, while on the sphere -- with $17$ more digits!  
 It is worth noting that such weather forecasts are  much more unreliable than on the torus or the Klein bottle. 

\begin{remark}\label{rem:lukatskii}
{\rm One should note that the predicted forecast trustworthiness is strongly affected by one's choice of $e^0_3$ as a ``trade wind'', cf.
~\cite{arnold_mathematical_1989}.  We argue however that $e^0_3$ is a more realistic approximation of the ``trade wind'' observed in the earths atmosphere than $e^0_2$ used before, see Figure \ref{fig:sphere_tradewind_new}.

In~\cite{lukatskii_curvature_1979} the results of calculations for the vector field
\[e^0_2 =\sqrt{\frac{15}{8\pi}}  \cos\theta \frac{\p}{\p \phi}, \]
taken as the ``trade wind'', are as follows: 
$\Ric(e^0_2)=-\frac{3}{S_{\mb S^2}}$, see Corollary~\ref{cor:ricci.s2.g.direction}; 
$\norm{v} = 1,$ and 
$|{v}|_{\text{rapid}} = \frac{1}{2} \sqrt{\frac{15}{8\pi }} =\frac{1}{2} \sqrt{\frac{15}{2 S_{\mb S^2} }} $. 
The sphere has the rescaling factor $c_{\mb S^2} = \frac{40~000\text{km}}{2\pi}$ and hence for $v =e^0_2$ we get
$$k_{\mb S^2} =  \frac{1\cdot \sqrt{\frac{3}{S_{\mb S^2 } } } }{ \frac{40~000}{2\pi} \frac{1}{2} \sqrt{\frac{15}{2 S_{\mb S^2} } }  } 30\cdot24\cdot100\log_{10}e  \simeq 6.2  \, ,$$ 
where an error grows by a factor $10^{kN}$ after $N$ months. In this case to predict the weather on the sphere for 2 months, one needs to know it today with $12$ more digits of accuracy. 

If for the ``average sectional curvature'' instead of $\Ric(e^0_2)$ one uses the estimate that $\Ric(e^0_2) \simeq \frac{1}{4} \inf C(e^0_2, e^m_l) = -\frac{15}{4\cdot 8\pi} $, as in~\cite{lukatskii_curvature_1979}, one obtains  $\tilde k_{\mb S^2}   \simeq 4.9$. This estimate for the sphere is more 
``inline" with Arnold's computations for the torus, where a similar estimate for the ``average sectional curvature'' was used, and one concludes
that to predict the weather on the sphere for 1 month, one needs to know it today with $5$ more digits of accuracy (and respectively with $10$ more digits for 2 months). 
}
\end{remark}

\begin{remark} 
{\rm Alternatively, instead of rescaling the equator's length, one can rescale the manifold's area. The area of the earth is  $S_{\text{earth}} = 5.1\cdot 10^8$ km$^2$. The corresponding rescaling factor $c_M$ for the length on $M$, is the square root of the ratios of the areas, i.e. 
$$
c_M = \sqrt{\frac{5.1\cdot 10^8\text{ km}^2}{S_M}}\,.
$$
The corresponding values of $S_M$ and $k$ are summarized here:
\begin{center}
\begin{tabular}{l | l l l l }
${ M } $ 	& $\mb K$ 		&	$\mb T$  		&  $\mb{RP}^2$  &  $\mb{S}^2$\\[7pt]
${S_M}$ 	& $2\pi^2 $		& $4\pi^2 $				& $2\pi $				& $4\pi $\\[7pt]
$k$ &  $2.6$ & $3.8$  & $6.0$  &  $8.5$
\end{tabular}
\end{center}
Note that the areas of the non-orientable manifolds and their orientation covers differ by a factor of 2, and hence the corresponding 
factors $c_M$ and, respectively, the exponents $k$ for the manifold and its cover differ by a factor $\sqrt 2$. 
}
\end{remark}


\appendix



\section{Lemma for Theorem~\ref{th:ric.klein} on the Ricci curvature}\label{app:technical.proofs}
\begin{lemma}\label{lem:klein.ricci.sum.over.f} Define $\lambda  =  {\rm arctan}\left(\frac{l_2}{l_1}\right)$, $\kappa =  {\rm arctan}\left(\frac{k_2}{k_1}\right)$ for indices $k,l$ and set
	\[f(\lambda) = \frac{\norm{k}^2}{8S_{\mb T}} \left(4\cos(2\lambda) \cos(2\kappa) - \cos(4\lambda) \cos(4\kappa)\right).  \] Then the average value of $f$ over all basis elements $\mcE_l \in \mathcal B$ vanishes, i.e. for $\mathcal{B}_R$ defined in~\eqref{eq:BR} 
\[ L =  \lim\limits_{R\to \infty}  \frac{1}{|\mathcal{B}_R|}  \sum_{\mcE_l \in \mathcal{B}_R}  f(\lambda) = 0. \]
\end{lemma}

\begin{proof}[Sketch of proof]
For each $R$, the number
$\frac{1}{\big|\mathcal{B}_R\big|} \sum_{\mcE_l \in \mathcal{B}_R} f(\lambda )$,
is  the mean value of $f$  evaluated over $\lambda =  {\rm arctan}(l_2/l_1)$ of $\mcE_l \in \mathcal{B}_R$.  The angle $\lambda \in [0, \pi/2]$ since $(l_1, l_2) \in \mb N_0^2$. If $\lambda$ were a continuous parameter in $[0, \pi/2]$, then the mean value of $f$ over $[0, \pi/2]$ would be zero. However, the sum is finite  for each $R$ and its value will depend on the (measure theoretic) density of angles $\lambda =  {\rm arctan}(l_2/l_1)$, i.e. the density of lines through the origin intersecting a point $(l_1, l_2)$ with integer coefficients.

\begin{figure}
\centering
\includegraphics[width=0.4\textwidth]{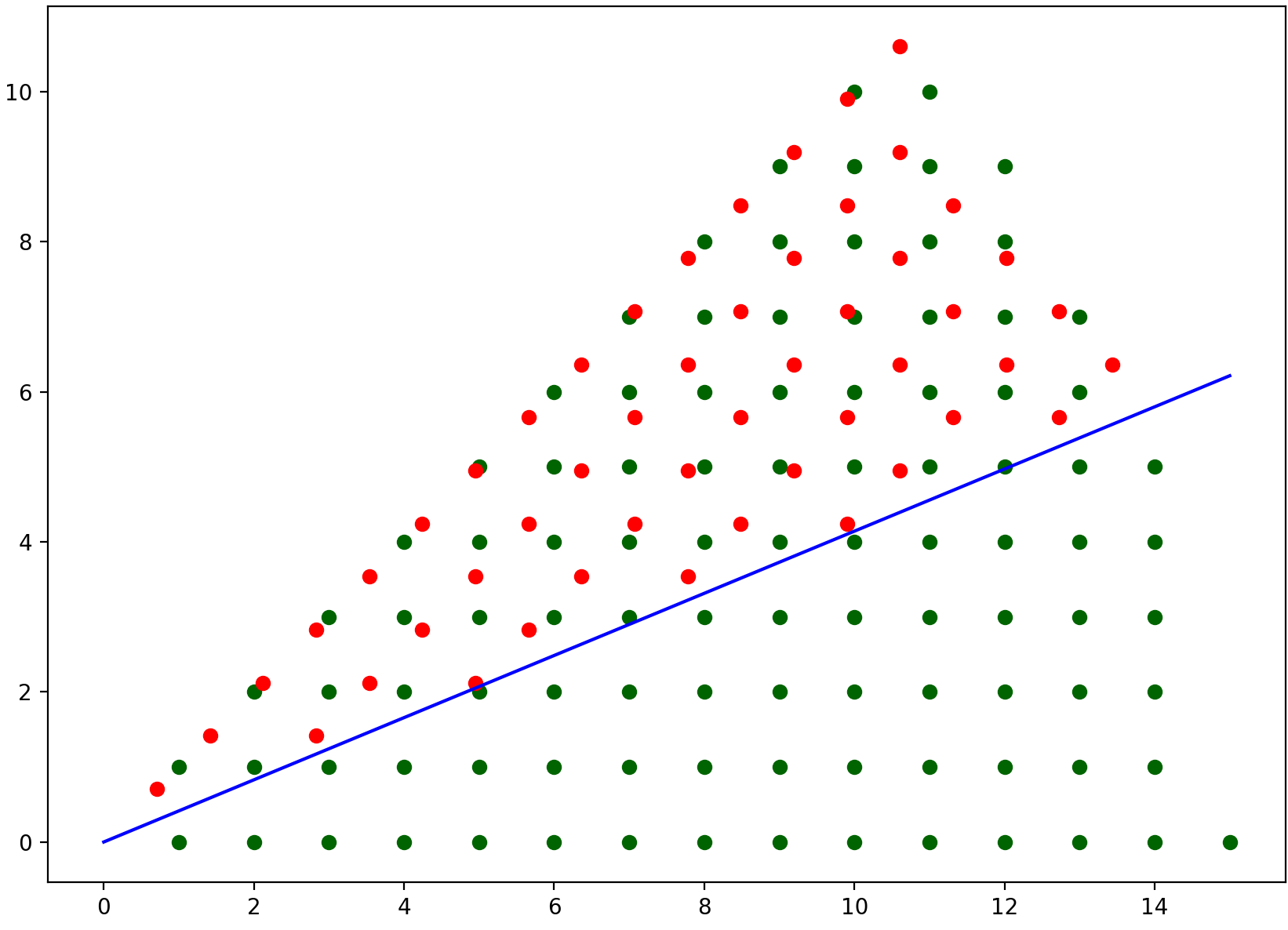}
\caption{ \tiny Reflection of the indices in $J^{\Im}$ through the ray $\theta = \pi/8$ for $\norm{l}\leq 15$ and $\lambda =  {\rm arctan}(l_2/l_1) \leq \pi/4$. The green points correspond to $l\in J^{\Im}$ and the red points to the points of $J^{\Im}$ with $\lambda < \pi/8$, reflected through the ray $\theta = \pi/8$, which is the ray in blue.}
\label{fig:pi8.reflection}
\end{figure}

We give a sketch of a more formal proof of the lemma. From Lemma~\ref{lem:basis.element.counting} the set $\mathcal{B}_R$ grows as $\frac{\pi}{2} R^2$. Subsets that grows as $\sim R$ can be removed from the sum over $f$, without changing the value of the sum, see Figure~\ref{fig:AR_complement}. Consider the indices $l = (l_1, l_2) \in \mb N^2$ such that $\mcE_l \in \mathcal B$. The indices are symmetric with respect to reflection through the ray $\lambda = \pi/4$ in $\mb N^2$ (modulo sets that grows as $\sim R$), mapping $ (l_1, l_2)$ into $(l_2, l_1)$. Since $\lambda' =  {\rm arctan}(l_1/l_2) = \pi/2 - \lambda$, then also $f(\lambda) + f(\lambda') = -2\cos(4\lambda)\cos(4\kappa)$. Thus
\begin{equation}\label{eq:first.estimate.cos.sum}
L = \lim\limits_{R\to \infty}   \frac{\norm{k}^2}{4S_{\mb T}\big|\mathcal{B}_R\big|}   \Big|   \sum_{ \substack{\mcE_l \in \mathcal B_R \\ 0\leq \lambda \leq \pi/4 } } \cos(4\lambda) \Big|.
\end{equation}
Now we reflect $l = (l_1, l_2)$  with $0\leq \lambda \leq \pi/4 $ through the ray $\lambda = \pi/8$ in $\mb N^2$ into points $\tilde{l}$, see Figure~\ref{fig:pi8.reflection}. The indices are not symmetric with respect to this reflection, but are uniquely paired with another index $j = (j_1, j_2)$ with $\pi/8 \leq \lambda_j =  {\rm arctan}(j_2/j_1) \leq \pi/4$, where $\mcE_j \in \mathcal B_R$ (modulo sets that grows as $\sim R$). The distance between the reflected point $\tilde{l}$ and its partner $j$ is less than $\frac{1}{\norm{j} }$. Hence
\[\big|\cos(4\lambda)  + \cos(4\lambda_j) \big|= \big|-\cos(4\tilde{\lambda} ) + \cos(4\lambda_j)\big| \leq 4\big|\tilde{\lambda} - \lambda_j\big| \leq 4\frac{1}{\norm{j} }.\]
Applying the above procedure to~\eqref{eq:first.estimate.cos.sum}, we obtain that 
\begin{equation*}
L = \lim\limits_{R\to\infty} \frac{\norm{k}^2}{4S_{\mb T}\big|\mathcal{B}_R\big|}   \sum_{ \substack{\mcE_l \in \mathcal{B}_R  \\ 0\leq \lambda \leq \pi/8 } } \left|\cos(4\lambda)  + \cos(4\lambda_j) \right| = 0,
\end{equation*}
since this is the average value of a function over a discrete set, where the limit of the function as $R\to \infty$ goes to zero. 

\end{proof}

\bibliography{NonorientCurv7.bib}

\end{document}